\documentclass[11pt]{amsart}
\usepackage{amssymb,latexsym}
\input xypic
  \xyoption{all}
\usepackage{amsmath, amsthm, amsfonts}

\usepackage[english]{babel}
\usepackage[latin1]{inputenc}
\usepackage[T1]{fontenc}
\usepackage[colorlinks,linkcolor=black,citecolor=black,hyperindex,dvipdfm]{hyperref}

\usepackage{graphicx}
\usepackage{diagbox,amssymb}
\usepackage{verbatim,enumerate}
\usepackage{xypic}

\newtheorem{thm}{Theorem}[section]
\newtheorem{prop}[thm]{Proposition}
\newtheorem{cor}[thm]{Corollary}

\newtheorem{lem}[thm]{Lemma}
{\theoremstyle{definition}

\newtheorem{rem}[thm]{Remark}}

\newcommand{\rb}{\mathbb{R}}

   \newcommand{\ba}{\begin{eqnarray}}
   \newcommand{\na}{\end{eqnarray}}
   \newcommand{\ban}{\begin{eqnarray*}}
   \newcommand{\nan}{\end{eqnarray*}}

\newcommand{\ind}{\operatorname{Ind}}

\newcommand{\cl}{\mathcal{C}}
\newcommand{\jl}{\mathcal{J}}

\newcommand{\ml}{\mathcal{M}}

\newcommand{\fl}{\mathcal{F}}
\newcommand{\zl}{\mathcal{Z}}
\newcommand{\ol}{\mathcal{O}}
\newcommand{\pl}{\mathcal{P}}

\newcommand{~}{\quad}
\newcommand{\cb}{\mathbb{C}}

\newcommand{\zb}{\mathbb{Z}}
\newcommand{\pb}{\mathbb{P}}
\newcommand{\bb}{\mathbb{B}}

\newcommand{\glt}{\geqslant}
\newcommand{\llt}{\leqslant}
\newcommand{\undl}{\underline}
\newcommand{\vep}{\varepsilon}

\begin{document}

\title[  Welschinger invariants of Blow-ups]
{ Welschinger invariants of Blow-ups of symplectic 4-manifolds}

\author{Yanqiao Ding$^*$ \& Jianxun Hu$^{**}$  }
\address{Department of Mathematics\\ Sun Yat-Sen University\\
                       Guangzhou, 510275 \\ P. R. China }
\email{dingyq6@mail2.sysu.edu.cn}
\thanks{${}^*$ Current address: Department of Mathematics, Zhengzhou University,
                 Zhengzhou, 450001, P. R. China}

\address{Department of Mathematics\\ Sun Yat-Sen University\\
                       Guangzhou, 510275 \\ P. R. China}
\email{stsjxhu@mail.sysu.edu.cn}
\thanks{${}^{**}$Partially supported by the NSFC Grant 11371381 and 11521101}


\begin{abstract}
Using the degeneration technique, one studies the behavior of Welschinger invariants under the blow-up, and obtains some
blow-up formulae of Welschinger invariants. One also analyses the variation of Welschinger invariants when replacing a pair
of real points in the real configuration by a pair of conjugated points, and reproves  Welschinger's wall crossing formula.
\end{abstract}

\date{\today}
\maketitle


\section{Introduction}

Traditional enumerative geometry asks certain question to which the expected answer is a number:
for instance, the number of lines incident with two points in the plane, or the number of twisted
cubic curves on a quintic 3-fold. For last two decades, the complex enumerative geometry of curves in algebraic varieties
has taken a new direction with the appearance of Gromov-Witten invariants and quantum cohomology. The core
of Gromov-Witten invariants is so called ``counting the numbers of rational curves". On real enumerative
geometry side, one expected a real version of Gromov-Witten invariants for long time. In 2005, Welschinger
\cite{wel2005a} first discovered such an invariant in dimensions 4 and 6, which was called the Welschinger invariant and revolutionized
the real enumerative geometry. Recently it was partially extended to higher dimensions, higher genera and descendant type,
see \cite{GZ, shustin2015, shustin2015a} and the references therein for the details. Itenberg-Kharlamov-Shustin \cite{iks2014} also extended the algebraic
definition of Welschinger invariants to all del Pezzo surfaces
and proved the invariance under deformation in algebraic setting.

After the Welschinger invariants were well defined for real symplectic 4-manifolds, the focus of the research on Welschinger invariants turned into
its computation for some manifolds and the understanding of its global structure.
Itenberg-Kharlamov-Shustin \cite{iks2003,iks2004,iks2007,iks2009,iks2013c,iks2013b} systematically studied the
Welschinger invariants of del Pezzo surfaces, including the lower  bounds
of the invariants, the logarithmic equivalence of Welschinger and Gromov-Witten invariants, positivity,
and Caporaso-Harris type formula for Welschinger invariants.
Brugall\'e-Mikhalkin \cite{bm2008,bm2007} provided a method to compute
Welschinger invariants via floor diagram. Using their method, Arroyo-Brugall\'{e}-de Medrano \cite{abl2011}
computed the Welschinger invariants in the projective plane. Based on the open analogues of Kontsevich-Manin axioms
and WDVV equation, Horev-Solomon \cite{hs2012} gave a recursive formula of Welschinger invariants of real
blow up of the projective plane .

Using the degeneration technique, Itenberg-Kharlamov-Shustin \cite{iks2013a}
studied the positivity and asymptotics of Welschinger invariants of
real del Pezzo surfaces of degree $\geq 2$,
and obtained some new real Caporaso-Harris type formulae and real analogues of
Abramovich-Bertram-Vakil formula. In \cite{bp2013,bp2014},  Brugall\'{e}-Puignau
applied the real version of symplectic sum formula to obtain a real version of
Abramovich-Bertram-Vakil formula in the symplectic setting. Combining their formula with degeneration formula
and the technique of floor diagrams relative to a conic, E. Brugall\'e \cite{bru2015} computed the Gromov-Witten invariants and Welschinger invariants of
some del Pezzo surfaces.

Another important issues in the study of Welschinger invariants are how to understand the behavior of Welschinger invariants
under geometric transformations and how to apply Welschinger invariants to investigate the geometry and topology of the underlying manifolds.
In \cite{bru2015,brugalle2016,bp2014,iks2013a}, the authors used the degeneration technique to study the properties of Welschinger invariants.
In particular, by locally modifying the real structure, E. Brugall\'{e} \cite{brugalle2016} proved very simple relations among Welschinger
invariants of real symplectic 4-manifolds differing by a real surgery along a real Lagrangian sphere. In fact, his real surgery is a kind of real
symplectic blow-up along a real Lagrangian sphere, see Section 4 of \cite{brugalle2016} for the details.

From the research of algebraic geometry and Gromov-Witten theory \cite{lr2001,li2002} ,
we know that the invariants obtained from the moduli spaces always have close
relationship with the birational transformation.
As well known, blow-up is the basic birational transformation.
The moduli space of genus zero curves is well behaved under blow-up.
The absolute value of Welschinger invariant provided a lower bound for the number of real pseudo-holomorphic
curves passing through a particular real configuration and representing a degree, whereas an upper bound is given by the
corresponding genus zero Gromov-Witten invariant.
Inspired by the works on Gromov-Witten invariants \cite{brugalle2016,hu2000, hu2001, hlr2008},
we will study the behavior of Welschinger invariants under real symplecitc blow-ups in this paper.

A \textit{real symplectic $4$-manifold} $(X, \omega, \tau)$, denoted by $X_\rb$, is a symplectic 4-manifold $(X,\omega)$
with an involution $\tau$ on $X$ such that $\tau^*\omega=-\omega$.
The fixed point set of $\tau$, denoted by $\rb X$, is called the \textit{real part} of $X$.
$\rb X$ is either empty or a smooth lagrangian submanifold of $(X,\omega)$.
An $\omega$-tamed almost complex structure $J$  is called $\tau$-compatible
if $\tau$ is $J$-antiholomorphic.
The space of all $\tau$-compatible almost complex structures on $X$ is denoted by $\rb\jl_\omega$.
Let $c_1(X)$ be the first Chern class of the symplectic manifold $(X,\omega)$.
Let $d\in H_2(X;\zb)$ be a homology class satisfying
$c_1(X)\cdot d>0$ and $\tau_*d=-d$.
 Let $L$ be a connected component of $\rb X$.
Assume $\undl x\subset X$ is a real configuration consisting of $r$ real points in $L$
and $s$ pairs of $\tau$-conjugated points in $X\backslash\rb X$,
where $r+2s=c_1(X)\cdot d-1$.
Fix a $\tau$-invariant class $F\in H_2(X\setminus L;\zb/2\zb)$.
Denote by $W_{X_\rb,L,F}(d,s)$ the Welschinger invariants.
For the simplicity of notations, we assume that $\rb X$ is connected and $F=0$.
In this situation, we denote $W_{X_\rb}(d,s)$ instead of $W_{X_\rb,L,F}(d,s)$.

Let $p:X_{a,b}\to X$ be the real symplectic blow-up of $X$ at $a$ real points
and $b$ pairs of $\tau$-conjugated points.
Denote by $p^!d=PDp^*PD(d)$, where $PD$ stands for the Poincar\'e duality.

From the point of geometric view, an intuitive observation is that one will get the same number
when we try to count the real rational pseudo-holomorphic curves in $X$ and its blow-up in Welschinger's way if the blown-up points are away from
the real configuration. This implies the following theorems.

\begin{thm}\label{thm:bpr}
Let $X_\rb$ be a compact real symplectic $4$-manifold,
$d\in H_2(X;\zb)$ such that $c_1(X)\cdot d>0$ and $\tau_*d=-d$.
Denote by $p: X_{1,0}\to X_\rb$ the projection of the real symplectic blow-up of $X_\rb$ at $x\in\rb X$.
Then
\begin{equation}\label{eq:bpr1}
W_{X_\rb}(d,s)  = W_{ X_{1,0}}(p^!d,s),
\end{equation}
\begin{equation}\label{eq:bpr2}
W_{X_\rb}(d,s)  =  W_{ X_{1,0}}(p^!d-[E],s),\,\,\,\,\,\text{ if }c_1(X)\cdot d-2s\glt 2,
\end{equation}
where $E$ denotes the exceptional divisor  and $p^!d=PDp^*PD(d)$.
\end{thm}

\begin{thm}\label{thm:bpc}
Let $X_\rb$ be a compact real symplectic $4$-manifold,
$d\in H_2(X;\zb)$ such that $c_1(X)\cdot d>0$ and $\tau_*d=-d$.
Suppose $y_1$, $y_2\in X\setminus\rb X$ is a $\tau$-conjugated pair, i.e., $\tau (y_1) = y_2$.
Denote by $p:X_{0,1}\to X_\rb$ the projection of the real symplectic blow-up of $X_\rb $ at $y_1$, $y_2$.
Then
\begin{equation} \label{eq:bpc}
W_{X_\rb}(d,s)  =  W_{X_{0,1}}(p^!d,s),
\end{equation}
\begin{equation}\label{eq:bpc-1}
W_{X_\rb}(d,s)  =  W_{X_{0,1}}(p^!d-[E_1]-[E_2],s-1), \,\,\,\,\,\, \mbox{ if } s \glt 1,
\end{equation}
where $E_1$, $E_2$ denote the exceptional divisors at $y_1$, $y_2$ respectively.
\end{thm}

From Theorem $\ref{thm:bpr}$ and Theorem $\ref{thm:bpc}$, it is easy to get

\begin{cor}\label{cor:bp}
Let $X_\rb$ be a compact real symplectic $4$-manifold,
$d\in H_2(X;\zb)$ such that $c_1(X)\cdot d>0$ and $\tau_*d=-d$.
Suppose $\undl{x}'\subset X$ is a real set consisting of $r'$ points in $\rb X$ and
$s'$ pairs of $\tau$-conjugated points in $X\backslash\rb X$ with $r'\llt r$, $s'\llt s$.
Denote by $p:X_{r',s'}\to X$  the projection of the real symplectic blow-up of $X$ at $\undl{x}'$.
Then
\begin{equation}\label{eq:mainresults}
W_{X_\rb}(d,s)  =  W_{X_{r',s'}}(p^!d,s),
\end{equation}
\begin{equation}\label{eq:mainresults-1}
W_{X_\rb}(d,s)  =  W_{X_{r',s'}}
(p^!d-\sum_{i=1}^{r'}[E_i]-\sum_{j=1}^{s'}([E_j']+[E_j'']),s-s'),
\end{equation}
where $E_i$, $E_j'$, $E_j''$ denote the exceptional divisors
corresponding to the real set $\undl x'$ respectively.
\end{cor}

    Welschinger \cite{wel2005a} introduced a new $\theta$-invariant to describe the dependence of Welschinger
invariants on the number of real points in the real configurations, and obtained a {\bf wall-crossing formula}, see Theorem 3.2 of \cite{wel2005a}. More
precisely, when replacing a pair of real fixed points in the same component of $\rb X$ by a pair of imaginary conjugate points,
twice  the $\theta$-invariant  is the difference of the respective invariants. In this paper, using the degeneration method,
we reprove Welschinger's wall-crossing formula and verify that Welschinger's $\theta$-invariants of $X_\rb$ are the Welschinger invariants of the real blow-up $X_{1,0}$ of the real symplectic manifold
at one real point.

\begin{thm}\label{thm:wall}
Let $X_\rb$ be a compact real symplectic $4$-manifold,
 $d\in H_2(X;\zb)$ such that $c_1(X)\cdot d\glt4$ and $\tau_*d=-d$.
Denote by $p: X_{1,0}\to X$ the projection of the real symplectic blow-up of $X_\rb $ at $x\in\rb X$.
If $s\geq 1$, then
\begin{equation}
W_{X_\rb}(d,s-1)=W_{X_\rb}(d,s)+2W_{X_{1,0}}(p^!d-2[E],s-1), \label{eq:bpr3}
\end{equation}
where $E$ denotes the exceptional divisor  and $p^!d =PDp^*PD(d)$.
\end{thm}

\begin{rem}\label{thm:bp}
The same argument as in the proofs of previous theorems generalizes the formulae (\ref{eq:mainresults}) and (\ref{eq:mainresults-1}) to
the general case that $\rb X$ is disconnected. More precisely, assume that  $X_\rb$ is a compact real symplectic $4$-manifold
and $\rb X$ is disconnected. Suppose $\undl{x}'\subset X$ is a real set made of $r'$ points in $L$ and
$s'$ pairs of $\tau$-conjugated points in $X$ with $r'\llt r$, $s'\llt s$.
Denote by $\tilde L$ the connected component of $\rb X_{r',s'}$ corresponding to $L$ .
If only one of the blown-up real points belongs to $L$, $\tilde L=L\sharp\rb P^2$.
We assume $F$ has a $\tau$-invariant compact representative
$\fl\subset X\setminus\undl{x}'$, and denote $\tilde F=p^!F$.
Denote by $p:X_{r',s'}\to X$ the projection of the real symplectic blow-up of $X$ at $\undl{x}'$.
Then
\begin{equation}\label{eq:main-1}
W_{X_\rb,L,F}(d,s) = W_{X_{r',s'},\tilde L,\tilde F}(p^!d,s),
\end{equation}
\begin{equation}\label{eq:main-2}
W_{X_\rb,L,F}(d,s) = W_{X_{r',s'},\tilde L,\tilde F}
(p^!d-\sum_{i=1}^{r'}[E_i]-\sum_{j=1}^{s'}([E_j']+[E_j'']),s-s'),
\end{equation}
\begin{equation}\label{eq:main-3}
W_{X_\rb,L,F}(d,s-1) = W_{X_\rb,L,F}(d,s)+2W_{X_{1,0},\tilde L,\tilde F}(p^!d-2[E],s-1),
\end{equation}
where $E_i$, $E_j'$, $E_j''$ denote the exceptional divisors
corresponding to the real set $\undl x'$ respectively.
\end{rem}

\medskip\noindent
{\bf Acknowledgment}: The first author would like to thank E. Brugall\'e
 for answering my question  about the behavior of pseudo-holomorphic
curves under symplectic sum.

\section{Preliminary}

\subsection{Real blow-ups of the projective plane}\label{subsec:real blow-up}

In this subsection, we consider how the standard real structure, i. e., the conjugate on $\mathbb{C}P^2$ induces a real structure on $\tilde{\mathbb{C}P}^2$.
For this purpose, we must distinguish the real points from other points in $\mathbb{C}P^2\setminus \mathbb{R}P^2$.

First of all, we review the blow-up of $\mathbb{C}P^2$ at a point $x$. Let $U$ be a neighborhood of $x$ with local coordinate $(z_1,z_2)$.
Denote
$$
     \pi : V:= \{((z_1,z_2), [w_1:w_2])\in U\times \mathbb{C}P^1\mid z_iw_j = z_jw_i\}  \longrightarrow U  \\
$$
the projection to $U$ via the first factor . There is a natural identification map $gl$ between $V\setminus E=\pi^{-1}(0)$ and $U\setminus \{x\}$, where $E= \pi^{-1}(0)$
is the exceptional divisor. We get the blow-up
$$
      \tilde{\mathbb{C}P}^2 = \mathbb{C}P^2\setminus \{x\} \bigcup_{gl} V.
$$

 For a real point $x\in \mathbb{R}P^2\subset\mathbb{C}P^2$, denote by $\tilde{\mathbb{C}P}^2_{1,0}$ the blow-up of $\mathbb{C}P^2$ at $x$.
We may choose a conjugation invariant neighborhood $U$ of $x$ in $\mathbb{C}P^2$ and the local coordinate $(z_1,z_2)$.

Define an involution $\tau : V\longrightarrow V$ as
$$
    \tau (((z_1,z_2), [w_1:w_2])) := ((\bar{z}_1, \bar{z}_2),[\bar{w}_1:\bar{w}_2]).
$$
It is easy to check that this involution coincides with the one induced on $V\setminus E$ by identification with $U\setminus \{x\}$. This implies that the standard
real structure on $\mathbb{C}P^2$ naturally induces a real structure on $\tilde{\mathbb{C}P}^2_{1,0}$ at a real point $x\in \mathbb{R}P^2$.

Since there is no conjugation-invariant neighborhood for the points in $\mathbb{C}P^2\setminus \mathbb{R}P^2$, to obtain a real structure on the blow-up, we need to blow up simultaneously
a pair of two conjugated points. Let $x_1$ and $x_2=\bar{x}_1$ be a pair of conjugated points in $\mathbb{C}P^2$. Choose a neighborhood $U_1$ of $x_1$ with local coordinate $(z_1,z_2)$, a neighborhood
$U_2= \bar{U_1}$ of $x_2$ with local coordinate $(\bar{z}_1, \bar{z}_2)$. Denote by $\tilde{\mathbb{C}P}^2_{0,1}$ the blow-up of $\mathbb{C}P^2$ at points $x_1$ and $x_2$ and $\pi$
the natural projection of the blow-up. Denote
\begin{eqnarray*}
     V_1 & := & \{((z_1,z_2), [w_1:w_2])\in U_1\times \mathbb{C}P^1\mid z_iw_j = z_jw_i\} = \pi^{-1}(U_1),\\
     & & \\
     V_2 & := & \{((z_1,z_2), [w_1:w_2])\in U_2\times \mathbb{C}P^1\mid z_iw_j = z_jw_i\} = \pi^{-1}(U_2).
\end{eqnarray*}
Define the involution $\tau : V_1\bigcup V_2 \longrightarrow V_1\bigcup V_2$ as
$$
       \tau(((z_1,z_2),[w_1:w_2])) := ((\bar{z}_1,\bar{z}_2), [\bar{w}_1:\bar{w}_2]).
$$
This involution coincides by construction with the one induced by the identification $\mathbb{C}P^2\setminus \{x_1,x_2\} = \tilde{\mathbb{C}P}^2_{0,1}\setminus \{\pi^{-1}(x_1), \pi^{-1}(x_2)\}$.
Therefore, we also can obtain a natural real structure on the blow-up $\tilde{\mathbb{C}P}^2_{0,1}$.

\subsection{Symplectic cut}\label{subsec:sympcut}

Lerman's symplectic cutting \cite{ler1995} is a simple and versatile operation on Hamiltonian $S^1$-manifolds.
Suppose that $X_0\subset X$ is an open codimension zero connected submanifold with a Hamiltonian $S^1$-action. Let  $H : X_0\longrightarrow \rb$
be a Hamiltonian function with $0$ as a regular value. If $H^{-1}(0)$ is a separating hypersurface of $X$, then we obtain two connected
manifolds $X_0^\pm$ with boundary $\partial X_0^\pm = H^{-1}(0)$, where the $+$ side corresponds to $H<0$.
Suppose further that $S^1$ acts freely on $H^{-1}(0)$ . Then the symplectic reduction $Z=H^{-1}(0)/S^1$ is canonically a symplectic manifold of dimension
2 less. Collapsing the $S^1$-action on $\partial X^\pm = H^{-1}(0)$, we obtain smooth closed  manifolds $\bar{X}^\pm$
containing respectively real codimension 2 submanifolds $Z^\pm = Z$ with opposite normal bundles. Furthermore $\bar{X}^\pm$
admits a symplectic structure $\bar{\omega}^\pm$ which agrees with the restriction of $\omega$ away from $Z$, and whose restriction to $Z^\pm$
agrees with the canonical symplectic structure $\omega_Z$ on $Z$ from symplectic reduction. The pair of symplectic manifolds $(\bar{X}^\pm, \bar{\omega}^\pm)$
is called the {\bf  symplectic cut} of $X$ along $H^{-1}(0)$.

This is neatly shown by considering $X_0\times \cb$ equipped with appropriate product symplectic structures and the product $S^1$-action on $X_0\times \cb$
where $S^1$ acts on $\cb$ by complex multiplication. The extended action is Hamiltonian if we use the standard symplectic structure $\sqrt{-1}dw\wedge d\bar{w}$
or its negative on the $\cb$ factor, see \cite{ler1995}. Denote by $\mu : X_0 \longrightarrow \mathbb{R}$ the moment map of the $S^1$-action on $X_0$, then
$\bar{X}^+ = \bar{X}_{\mu \leq \epsilon}$, $\bar{X}^- = \bar{X}_{\mu\geq \epsilon}$.

The normal connected sum operation \cite{gom1995,mw1994} or the fiber sum operation is the inverse operation of the symplectic cut.
Given two symplectic manifolds containing symplectomorphic codimension 2 symplectic submanifolds with opposite normal bundles, the normal connected sum operation produces
a new symplectic manifold by identifying the tubular neighborhoods.

Notice that we can apply the normal connected sum operation to the pairs $(\bar{X}^+, \bar{\omega}^+, Z^+)$ and $(\bar{X}^-, \bar{\omega}^-, Z^-)$ to recover $(X,\omega)$.

According to McDuff \cite{m1991}, the blow-up operation in symplectic geometry amounts to a
removal of an open symplectic ball followed by a collapse of some boundary directions.
In fact, we may apply the symplectic cut to construct the blow-up of symplectic manifold $X$ at a point $p$.
By the symplectic neighborhood theorem, take $X_0$ be a symplectic ball of radius $\epsilon_0$ centered at $p$ with complex
coordinates $(z_1, \cdots, z_n)$, where $\dim X = 2n$. Consider the Hamiltonian $S^1$-action on $X_0$ by
complex multiplication. Fix $\epsilon$ with $0<\epsilon <\epsilon_0$ and consider the moment map
$$
      H(u) = \sum_{i=1}^n |z_i|^2 - \epsilon, \,\,\,\,\, u\in X_0.
$$
Write the hypersurface $P= H^{-1}(0)$ in $X$  corresponding to the sphere with radius $\epsilon$. We cut $X$ along $P$ to obtain
two closed symplectic manifolds $\bar{X}^+$ and $\bar{X}^-$, one of which is $\cb P^n$. Following the notations of \cite{hu2000,lr2001},
we denote   $\bar{X}^+=\cb P^n$, $\bar{X}^-=\tilde X$  is the {\bf symplectic blow-up} of $X$.

\subsection{Real symplectic cut}\label{subsec:realsympcut}

Let $(X,\omega_X,\tau_X)$ and $(Y,\omega_Y,\tau_Y)$ be two real compact symplectic manifolds.
Suppose $V\subset X$ and $V\subset Y$ are real symplectic hypersurfaces so that $e( N_{V|X})+e( N_{V|Y})=0$,
where $N_{V|X}$, $N_{V|Y}$ are the normal bundles of $V$ in $X$ and $Y$ respectively.
Denote by $\omega_V$ the symplectic form $\omega_X|_V=\omega_Y|_V$ on $V$
and by $\tau_V$ the real structure $\tau_X|_V=\tau_Y|_V$.
There is a real structure $\tau_\sharp$ on $X\sharp_VY$ induced by the real structures
$\tau_X$, $\tau_Y$.
Suppose $\pi:\zl\to \Delta$ is the symplectic sum(c.f.\cite{gom1995,ip2004,mw1994}).
Equip the disc $\Delta$ with the complex conjugate.
There is a real structure $\tau_\zl$ on $\zl$ induced by $\tau_X$ and $\tau_Y$
such that the map $\pi:\zl\to \Delta$ is real.
See \cite{brugalle2016} or \cite{bp2014} for more about real symplectic sum in dimention $4$.

Let $(X,\omega,\tau)$ be a real symplectic manifold.
 Assume $H:X\to\rb$ is a $\tau$-invariant smooth Hamiltonian i.e., $H\circ\tau=H$,
then we call $H$ a \textit{real Hamiltonian}(c.f.\cite{s2010}).
A Hamiltonian circle action on $(X,\omega,\tau)$ is a $1$-parameter subgroup $\rb\to Symp(X):t\mapsto\psi_t$
of symplectomorphisms of $X$ which is $2\pi$-periodic, i.e. $\psi_{2\pi}=id$, and which is the integral of a Hamiltonian
vector field $X_H$. The Hamiltonian function $H:X\to\rb$ in this case is called the moment map of the action.
If the Hamiltonian circle action on $(X,\omega,\tau)$ satisfies
\begin{equation}\label{eq:rhcaction}
\psi_{2\pi-t}\circ\tau=\tau\circ\psi_t
\end{equation}
for all $t\in[0,2\pi]$, we call it a \textit{real Hamiltonian circle action}. The moment map of a real Hamiltonian
circle action is a real Hamiltonian.

Let $(X,\omega,\tau)$ be a real symplectic manifold with a real Hamiltonian circle action.
Suppose that $\mu:X\to\rb$ is a real moment map.
Let $(\mu^{-1}(0)/S^1,\omega_\mu)$ be the symplectic reduction.
There is a natural real structure $\tau_\mu$
on $\mu^{-1}(0)/S^1$ induced by $\tau$ on $X$.
Define
\begin{align*}
\tau_\mu:\mu^{-1}(0)/S^1&\to \mu^{-1}(0)/S^1\\
[x]&\mapsto [\tau(x)].
\end{align*}
Suppose $x$, $y\in \mu^{-1}(0)$ such that $[x]=[y]$,
then there is a $t\in[0,2\pi]$ such that $\psi_t(x)=y$.
By equation $(\ref{eq:rhcaction})$,
$\psi_{2\pi-t}(\tau(x))=\tau\circ\psi_t(x)=\tau(y)$.
Therefore, $[\tau(x)]=[\tau(y)]$ and $\tau_\mu$ is well defined.
Obviously,
the reduced space $(\mu^{-1}(0)/S^1,\omega_\mu,\tau_\mu)$
is a real symplectic manifold.

Then we can find that the real symplectic manifold
$(\mu^{-1}(\vep)/S^1,\omega_\mu,\tau_\mu)$ is a real symplectic manifold embedded in both
$\bar X_{\mu\glt\vep}$ and $\bar X_{\mu\llt\vep}$ as a codimension $2$
real symplectic submanifold but with opposite normal bundles.
The pair of real symplectic manifolds $\bar X_{\mu\glt\vep}$, $\bar X_{\mu\llt\vep}$
is called the {\bf real symplectic cut} of $X$ along $\mu=\vep$.

\begin{rem}
Let $X_0\subset X$ be an open codimension zero real submanifold equipped with
real Hamiltonian $S^1$ action with a real proper momentum map $\mu:X_0\to\rb$.
Suppose that $\mu$ achieves its maximal value $c$ at a single point $p\in\rb X$.
For a sufficient small $\vep$, $p$ is the only critical point in the
set $X_{\mu>c-\vep}=\{x\in X_0|c-\mu(x)<\vep\}$.
For all $0<\delta<\vep$, the real symplectic manifold $\bar X_{\mu\llt c-\delta}$
is the real blow-up of $X_0$ at $p$ by a $\delta$ amount.
We define $\bar{X}^+:=\bar X_{\mu\glt c-\delta}$ and
$\bar{X}^-:=(X-X_0)\cup\bar X_{\mu\llt c-\delta}$.
We have $\bar{X}^+=\cb P^n$, $\bar{X}^-=X_{1,0}$, where $X_{1,0}$ is the real blow-up of
$X$ at a real point.
Suppose that $\mu$ achieves its maximal value $c$ at a single $\tau$-conjugated pair
$p_1$, $p_2\in X\setminus\rb X$.
Then $\bar{X}^+=\cb P^n\amalg\cb P^n$, $\bar{X}^-=X_{0,1}$, where $X_{0,1}$ is the real blow-up of
$X$ at a pair of $\tau$-conjugated points.
\end{rem}

\subsection{Welschinger invariants}\label{subsec:weldef}
Let $(X,\omega,\tau)$ be a compact real symplectic $4$-manifold.
Assume that the first Chern class $c_1(X)$ of the symplectic manifold $(X,\omega)$
is not a torsion element and let $d\in H_2(X;\zb)$ be a homology class satisfying
$c_1(X)\cdot d>0$ and $\tau_*d=-d$.
Let $\undl x=(x_1,...,x_m)$ be an ordered set of distinct points of $X$
such that $\tau(\undl x)=\undl x$.
Such a set is called a \textit{real configuration} of points.
Let $\sigma(\tau)$ be the order two permutation of $\{1,...,m\}$ induced by $\tau$.
Let $S$ be an oriented $2$-sphere, $\jl_S$ be the space of complex structure
of class $C^l$ of $S$ compatible with its orientation and $\undl z=(z_1,...,z_m)$
be $m$ distinct points on $S$. Define $\pl^d(\undl x)$ be the set of pseudo-holomorphic maps from $S$ to $X$ which
 pass through $\undl x$ and represent class $d$. Let $\pl^*(\undl x)$ be the subspace of $\pl^d(\undl x)$
consisting of simple maps.

 Denote by $\mbox{\it Diff}(S,\undl z)$ the group of diffeomorphisms of class $C^{l+1}$ of $S$,
which either preserve the orientation and fix $\undl z$,
or reverse the orientation and induce the permutation on $\undl z$ associated to $\tau$.
Let $\mbox{\it Diff}^+(S,\undl z)$(resp. $\mbox{\it Diff}^-(S,\undl z)$) be the subgroup of $\mbox{\it Diff}(S,\undl z)$
of orientation preserving diffeomorphisms (resp. its complement in $\mbox{\it Diff}(S,\undl z)$).
The group $\mbox{\it Diff}(S,\undl z)$ acts on $\pl^d(\undl x)\times\jl_S\times\jl_\omega$ by
\begin{equation*}
\phi(u,J_S,J)=
\left\{ \begin{array}{ll}
(u\circ\phi^{-1},(\phi^{-1})^* J_S,J),\,\,\,\,\mbox{if}\,\, \phi\in \mbox{Diff}^+(S,\undl z),\\
\\
(\tau\circ u\circ\phi^{-1},(\phi^{-1})^*J_S, \tau^* J), \,\,\,\, \mbox{if}\,\, \phi\in \mbox{Diff}^-(S,\undl z),
        \end{array} \right.
\end{equation*}
Denote by $\ml^d(\undl x)$ the quotient of $\pl^*(\undl x)$ by the action of
$\mbox{Diff}^+(S,\undl z)$. Let $\pi:\ml^d(\undl x)\to \jl_\omega$ be the projection.

\begin{prop}{\cite[Proposition 1.8]{wel2005a}}
The space $\ml^d(\undl x)$ is a separable Banach manifold of class $C^{l-k}$.
The projection $\pi:\ml^d(\undl x)\to\jl_\omega$ is Fredholm of index
$\ind_\rb(\pi)=2(c_1(X)\cdot d-1-m)$.
\end{prop}

The manifold $\ml^d(\undl x)$ is equipped with an $\zb/2\zb$ action.
Let $\rb\ml^d(\undl x)$ denote the fixed point set of this action.
$\pi_\rb:\rb\ml^d(\undl x)\to\rb\jl_\omega$ is the projection.

\begin{prop}{\cite[Proposition 1.9]{wel2005a}}
The projection $\pi_\rb:\rb\ml^d(\undl x)\to\rb\jl_\omega$ is Fredholm
of index $\ind_\rb(\pi_\rb)=c_1(X)\cdot d-1-m$.
\end{prop}

Suppose $\rb X$ is connected.
Let $c_1(X)\cdot d-1=r+2s$, $\undl x\subset X$ be a real configuration consisting of $r$ real points in $\rb X$
and $s$ pairs of $\tau$-conjugated points in $X\backslash\rb X$.
In this case, we denote $\cl(d,\undl x,J)=\ml^d(\undl x)$,
$\rb\cl(d,\undl x,J)=\mathbb{R} \ml^d(\undl x)$.

\begin{prop}\cite{wel2005a}
Let $J\in\rb\jl_\omega$ be generic, the set $\cl(d,\undl x,J)$ is finite.
Moreover, these curves are all irreducible and have only transversal double
points as singularities. The total number of double points of curve $C$ in $\cl(d,\undl x,J)$
equal to $\delta=\frac{1}{2}(d^2-c_1(X)\cdot d+2)$.
\end{prop}

Assume $C\in\rb\ml^d(\undl x)$.
The real double points of $C$ are of two different kinds.
They are either non-isolated or isolated.
A real double point is called \textit{non-isolated}
if it is the local intersection of two real branches.
The real nodal point which is the local intersection of two complex conjugated branches
is called \textit{isolated}.
The \textit{mass} $m_X(C)$ is defined to be the number of its isolated real nodal points
which satisfies $0\llt m(C)\llt\delta$.
The integer
$$
W_{X_\rb}(d,s)=\sum_{C\in\rb\cl(d,\undl{x}, J)}(-1)^{m_X(C)}
$$
neither depends on the choice of $J$, $\undl{x}$, nor on the deformation class of
$X_\rb$ (c.f.\cite{wel2003,wel2005a}).
These numbers are called \textit{Welschinger invariants} of $X_\rb$.

When the real part $\rb X$ is disconnected,
let $L$ be a connected component of $\rb X$.
Suppose $f:S\to X$ is an immersed real rational $J$-holomorphic curve in $X$
such that $f(\rb S)\subset L$, for a $J\in\rb\jl_\omega$.
Denoting by $S^+$ a half of $S\setminus\rb S$,
$f(S^+)$ defines a class $[f(S^+)]$ in $H_2(X,L;\zb/2\zb)$.
There exists a well defined pairing
$$
H_2(X,L;\zb/2\zb)\times H_2(X\setminus L;\zb/2\zb)\to\zb/2\zb
$$
given by the intersection product modulo 2.
Fix a $\tau$-invariant class $F\in H _2(X\setminus L;\zb/2\zb)$.
Define the $(L,F)$-mass of $f$ as
$$
m_{L,F}(f)=m_L(f)+[f(S^+)]\cdot F,
$$
where $m_L(f)$ is the number of real isolated nodes of $f$ in $L$.
$m_{L,F}(f)$ does not depend on the chosen half of $S\setminus\rb S$.

Given $J\in\rb\jl_\omega$, the set $\rb\cl(d,\undl{x}, J)$ consists of real rational
$J$-holomorphic curves $f:S\to X$ in $X$ realizing the class $d$, passing through $\undl{x}$,
and such that $f(\rb S)\subset L$.
Note that if $r\glt1$, the condition $f(\rb S)\subset L$ is always satisfied.
Itenberg, Kharlamov and Shustin  \cite{iks2013b} observed  that
the integer
$$
W_{X_\rb,L,F}(d,s)=\sum_{C\in\rb\cl(d,\undl{x}, J)}(-1)^{m_{L,F}(C)}
$$
neither depends on the choice of $J$, $\undl{x}$, nor on the deformation class of
$X_\rb$. Note that if $F=[\rb X\setminus L]$,
$W_{X_\rb,L,F}(d,s)$ is the original Welschinger invariant.
For the simplicity of notation, we assume $\rb X$ is connected and $F=0$.
At this situation, we denote $W_{X_\rb}(d,s)$ instead of $W_{X_\rb,L,F}(d,s)$.

\subsection{Curves with tangency conditions}\label{subsec:renum}

When we use the degeneration technique to study the behavior of curves under blow-up, we
need to deal with curves with tangency conditions. In this subsection, we review some basics on curves
with tangency conditions. We use Section 2.1 of \cite{bp2014} as the reference.

Two $J$-holomorphic maps
$f_1:C_1\to X$ and $f_2:C_2\to X$ are said to be isomorphism if there is
a biholomorphism $\phi:C_1\to C_2$ such that $f_1=f_2\circ\phi$.
In the following, maps are always considered up to isomorphism.
Given a vector $\alpha=(\alpha_i)_{1\llt i<\infty}\in\zb_{\glt0}^\infty$,
we use the notation:
$$
|\alpha|=\sum_{i=1}^{+\infty}\alpha_i,  \,\,\,\,\,
I\alpha=\sum_{i=1}^{+\infty}i\alpha_i.
$$
For $k\in \mathbb{Z}_{\geq 0}$ and $\alpha = (\alpha_i)_{1\leq i <\infty}$, denote $k\alpha := (k\alpha_i)_{1\leq i <\infty}$.
Let $\delta_i$ denote the vector in $\zb_{\glt0}^\infty$ whose all coordinates are
equal to $0$ except the $i$th one which is equal to $1$.

Let $(X,\omega)$ be a compact and connected $4$-dimensional symplectic manifold,
and let $V\subset X$ be an embedded symplectic curve in $X$.
Let $d\in H_2(X;\zb)$ and $\alpha$, $\beta\in\zb_{\glt0}^\infty$ such that
$$
I\alpha+I\beta=d\cdot[V].
$$
Choose a configuration $\undl x=\undl x^\circ\sqcup\undl x_V$ of points in $X$,
with $\undl x^\circ$ a configuration of $c_1(X)\cdot d-1-d\cdot[V]+|\beta|$
points in $X\setminus V$, and $\undl x_V=\{p_{i,j}\}_{0<j\llt\alpha_i,i\glt1}$ a configuration
of $|\alpha|$ points in $V$. Given an $\omega$-tamed almost complex structure $J$ on $X$ such that
 $V$ is $J$-holomorphic,  denote by $\cl^{\alpha,\beta}(d,\undl x,J)$ the set
of rational $J$-holomorphic maps $f:\cb P^1\to X$ such that

$\bullet$ $f_*[\cb P^1]=d$;

$\bullet$ $\undl x\subset f(\cb P^1)$;

$\bullet$ $V$ does not contain $f(\cb P^1)$;

$\bullet$ $f(\cb P^1)$ has order of contact $i$ with $V$ at each points $p_{i,j}$;

$\bullet$ $f(\cb P^1)$ has order of contact $i$ with $V$ at exactly $\beta_i$ distinct points on $V\setminus\undl x_V$.

The set of simple maps in $\cl^{\alpha,\beta}(d,\undl x,J)$ is $0$-dimensional if the almost
complex structure $J$ is chosen to be generic. However, $\cl^{\alpha,\beta}(d,\undl x,J)$
might contain components of positive dimension corresponding to non-simple maps.

\begin{lem}{\cite[Lemma 11]{bp2014}}\label{lem:simap}
Suppose that $\beta=(d\cdot[V],0,\cdots )$ and $\alpha=0$, or $\beta=(d\cdot[V]-1,0,\cdots)$ and $\alpha=(1,0,\cdots)$.
Then for a generic choice of $J$, the set $\cl^{\alpha,\beta}(d,\undl x,J)$ only contains
simple maps.
\end{lem}

\begin{prop}{\cite[Proposition 13]{bp2014}}\label{prop:finite}
Suppose that $V$ is an embedded symplectic sphere with $[V]^2=-1$, and that $|\beta|\glt d\cdot[V]-1$.
Then for a generic choice of $J$, the set $\cl^{\alpha,\beta}(d,\undl x,J)$ contains finitely many
simple maps. As a consequence, the set
$$
\cl^{\alpha,\beta}_*(d,\undl x,J)=\{f(\cb P^1)|(f:\cb P^1\to X)\in\cl^{\alpha,\beta}(d,\undl x,J)\}
$$
is also finite.
\end{prop}

In particular, suppose that $X=\cb P^2$, $V=H\subset \mathbb{C}P^2$ is the hyperplane in $\cb P^2$,
and $|\undl x^0|=1$, the set $\cl^{\alpha,\beta}(d,\undl x,J)$
is always finite and made of simple maps.

\begin{lem}\label{lem:p2set}
Suppose that $X=\cb P^2$ and $V=H\subset \mathbb{C}P^2$ is the hyperplane in $\mathbb{C}P^2$. Then the set $\cl^{\alpha,\beta}(d,\undl x,J)$
with $|\undl x^0|=1$ is empty for a generic choice of $J$, except
$\cl^{\delta_1,0}([H],\{p\}\cup\undl x_V,J)$ which contains an unique element.
Moreover, this unique element is an embedding.
\end{lem}

\begin{proof}
Suppose that $d = a[H], a\geq 0$.
Since $c_1(X)=3[H]$, we have
$$
c_1(X)\cdot(a[H])-1-(a[H])\cdot H+|\beta|=2a-1+|\beta|.
$$
Suppose that $2a-1+|\beta|=1$ and $\cl^{\alpha,\beta}(a[H],\undl x,J)\neq\emptyset$,
where $|\undl x^0|=1$.
From $I\alpha+I\beta=d\cdot [V] = a[H]\cdot [V] =a$, we can get $a[H]\cdot [H]=a\glt|\beta|$.
The intersection number $a$ of $J$-holomorphic curve $f\in\cl^{\alpha,\beta}(a[H],\undl x,J)$
with $V$ must satisfy $a\glt0$. So we obtain $a=1$, $|\beta|=0$.

If $\cb P^2$ is equipped with the symplectic form $\omega_{FS}$ and
its standard complex structure $J_{st}$, it is well-known that
$\cl^{\delta_1,0}([H],\{p\}\cup\undl x_V,J_{st})$ consists of a unique element.
When $\omega$ and $J$ are both varied, the corresponding set still contains at least one element.
If there are two distinct curves $C_1$ and $C_2$ in $\cl^{\delta_1,0}([H],\{p\}\cup\undl x_V,J)$,
then both $C_1$ and $C_2$ pass through ${p}\bigcup x_V$ which contains at least two points.
Therefore,  by the positivity of intersections, $C_1\cdot C_2 = 2$. This is impossible because $C_1\cdot C_2 = [H]\cdot [H] =1$.

 This contradiction implies that $\cl^{\delta_1,0}(H,\{p\}\cup\undl x_V,J)$ also consists of a unique element.
Thanks to the adjunction formula, this $J$-holomorphic curve is an embedding curve.
\end{proof}

Let $\tilde{\mathbb{C}P}^2$ be the blow-up of $\cb P^2$ at a point,
and $E$  the exceptional divisor.
It is easy to see $\tilde{\mathbb{C}P}^2\cong\pb_E(\ol(-1)\oplus\ol)$.
Let $E_0:=\pb_E(0\oplus\ol)$, $E_\infty:=\pb_E(\ol(-1)\oplus0)$.
$E_0$ and $E_\infty$ are two distinguished non-intersecting sections of $\pb_E(\ol(-1)\oplus\ol)$.
One computes easily that
$$
[E_\infty]^2=-[E_0]^2=1.
$$
The group $H_2(\tilde{\mathbb{C}P}^2,\zb)$ is the free abelian group generated by $[E_\infty]$ and $[F]$,
where $F$ is a fiber of $\pb_E(\ol(-1)\oplus\ol)\to E$.
The first Chern class of $\tilde{\mathbb{C}P}^2$ is given by
$$
c_1(\tilde{\mathbb{C}P}^2)=3[E_\infty]-[E_0]=2[E_\infty]+[F].
$$

In $X=\tilde{\mathbb{C}P}^2\cong\pb_E(\ol(-1)\oplus\ol)$, $V=E_\infty$, and $|\undl x^0|=0$, the set
$\cl^{\alpha,\beta}(d,\undl x,J)$ is always finite and made of simple maps.

\begin{lem}\label{lem:tp2set}
Suppose that $X=\tilde{\mathbb{C}P}^2$ and $V=E_\infty$. Then the set $\cl^{\alpha,\beta}(d,\undl x,J)$
with $|\undl x^0|=0$ is empty for a generic choice of $J$, except
$\cl^{\delta_1,0}([F],\undl x_V,J)$ which contains an unique element.
Moreover, the unique element is an embedding.
\end{lem}

\begin{proof}
Suppose $ d= a[E_\infty ] + b [F]$. Since $c_1(\tilde{\mathbb{C}P}^2)=2[E_\infty]+[F]$, we get
$$
c_1(X)\cdot d-1-d\cdot[E_\infty]+|\beta|=2a+b-1+|\beta|.
$$
Suppose that $2a+b-1+|\beta|=0$ and $\cl^{\alpha,\beta}(a[E_\infty]+b[F],\undl x_V,J)\neq\emptyset$.
Since $|\beta|\llt a+b$ and $|\beta|\glt0$, we have $a+b\glt0$.
By the positivity of intersection, we obtain
\begin{eqnarray*}
d\cdot [E_0] & = & (a[E_\infty]+b[F])\cdot([E_0])=b\glt0\\
& & \\
d\cdot [F] & = & (a[E_\infty]+b[F])\cdot[F]=a\glt0.
\end{eqnarray*}
We can deduce that $a=|\beta|=0$, $b=|\alpha|=1$.

The proof of the reminder of this lemma is similar to that in the proof of Lemma \ref{lem:p2set}.

\end{proof}

\section{Blow-up formula of Welschinger invariants}

\subsection{Blow-up formula at a real point}\label{subsec:bpform}
In this subsection, we consider the behavior of Welschinger invariants under the blow-up of symplectic
$4$-manifold at a real point.

Let $X$ be a compact real symplectic $4$-manifold. Perform a real symplectic cut on $X$ at the real point $x\in\rb X$.
We can get two real symplectic $4$-manifolds $\bar{X}^+\cong\pb^2$ and $\bar{X}^-\cong X_{1,0}$
which contain a common real symplectic submanifold $V$.
In $\bar{X}^+$, $V\cong H$  is the hyperplane in $\pb^2$.
In $\bar{X}^-$, $V\cong E$ is the exceptional divisor in $X_{1,0}$.

Let $\pi:\zl\to\Delta$ be the real symplectic sum of $\bar{X}^+$ and $\bar{X}^-$ along $V$,
$d\in H_2(\zl_\lambda;\zb)$. Choose $\undl x(\lambda)$ a set of $c_1(X)\cdot d-1$ real
symplectic sections $\Delta\to\zl$ such that $\undl x(0)\cap V=\emptyset$.
Choose an almost complex structure $J$ on $\zl$ tamed by $\omega_\zl$, which restrict to an
almost structure $J_\lambda$ tamed by $\omega_\lambda$ on each fiber $\zl_\lambda$, and generic with
respect to all choices we made.

Let $X_\sharp= \bar{X}^+\cup_V \bar{X}^-$.
Denote $\cl(d,\undl x(0),J_0)$ to be the set $\{\bar f:\bar C\to X_\sharp\}$ of limits, as stable maps,
of maps in $\cl(d,\undl x(\lambda), J_\lambda)$ as $\lambda$ goes to $0$,
where $\cl(d,\undl x(\lambda), J_\lambda)$ is the set of all irreducible rational $J$-holomorphic curves in
$(\zl_\lambda,\omega_\lambda,J_\lambda)$ passing through all points in $\undl x(\lambda)$, realizing the
class $d$. From \cite[Section $3$]{ip2004}, we know $\bar C$ is a connected nodal rational curve such that:

$\bullet$ $\undl x(0)\subset\bar f(\bar C)$;

$\bullet$ any point $p\in\bar f^{-1}(V)$ is a node of $\bar C$ which is the intersection of two
irreducible components $\bar C'$ and $\bar C''$ of $\bar C$, with $\bar f(\bar C')\subset X^+$
and $\bar f(\bar C'')\subset X^-$;

$\bullet$ if in addition neither $\bar f(\bar C')$ nor $\bar f(\bar C'')$ is entirely mapped into $V$,
then the multiplicities of intersection of both $\bar f(\bar C')$ and $\bar f(\bar C'')$ with $V$ are equal.

Given an element
$\bar f:\bar C\to X_\sharp$ of $\cl(d,\undl x(0),J_0)$, denote by $C_*$, $*=+,-$, the union of the irreducible components of
$\bar C$ mapped into $\bar{X}^*$.

\begin{prop}\label{prop:sum11}
Assume $\undl x(0)\cap \bar{X}^+$ contains at most one point,
$\undl x(0)\cap \bar{X}^-\neq\emptyset$ if $\undl x(0)\cap \bar{X}^+\neq\emptyset$.
Then for a generic $J_0$,
the set $\cl(d,\undl x(0),J_0)$ is finite, and only depends on $\undl x(0)$ and $J_0$.
Given an element $\bar f:\bar C\to X_\sharp$  of $\cl(d,\undl x(0),J_0)$,
the restriction of $\bar f$ to any component of $\bar C$ is a simple map,
and no irreducible component of $\bar C$ is entirely mapped into $V$.
Moreover the following are true:
\begin{enumerate}
\item If $\undl x(0)\cap \bar{X}^+=\emptyset$, then $C_+$ is empty.
The curve $C_-$ is irreducible,
and $\bar f|_{C_-}$ is an element of $\cl^{0,0}(p^!d,\undl x(0)\cap \bar{X}^-, J_0)$.
The map $\bar f$ is the limit of a unique element of
$\cl(d,\undl x(\lambda),J_\lambda)$ as $\lambda$ goes to $0$.
\item If $\undl x(0)\cap \bar{X}^+=\{p\}$, then $C_+$ is irreducible and
$\bar f(C_+)$ realizes class $[H]$.
The curve $C_-$ is irreducible and $\bar f|_{C_-}$ is an element of
$\cl^{0,\delta_1}(p^!d-[E],\undl x(0)\cap \bar{X}^-,J_0)$. The map $\bar f$ is the limit of a unique element of
$\cl(d,\undl x(\lambda),J_\lambda)$ as $\lambda$ goes to $0$.
\end{enumerate}
\end{prop}

\begin{proof}

From Example 11.4 and Lemma 14.6 of \cite{ip2004}, we know that no component of $\bar C$ is entirely mapped into $V$, also see  \cite{bp2014}.

Note that $[E]^2=-1$ in the real blow-up $\bar{X}^- = X_{1,0}$, and  $c_1(\bar X^-)\cdot[E]=1$.
Suppose $\bar f_*[C_+]=a[H]$, $a\geq 0$, $\bar f_*[C_-]=p^!d-b[E]$, $b\geq 0$.
Then we have
$$
a=\bar f_*[C_+]\cdot[H]=(p^!d-b[E])\cdot[E]=b.
$$

Since  $\undl x(0)\cap \bar X^+$ contains at most one point,
we will consider the two cases separately.

{\bf Case I:}  $\undl x(0)\cap \bar X^+=\emptyset$.

In this case, we know that $\bar f(C_{-})$ passes through all the $c_1(X)\cdot d-1$ points in
$\undl x(0)\cap \bar{X}^-$ and realizes the class $p^!d-b[E]$ in $H_2(\bar{X}^-;\zb)$.
Suppose $C_{-}$ consists of irreducible components $\{C_{-i}\}_{i=1}^m$,
and there are $0 \llt k\llt m$ irreducible components $\{C_{-i}\}_{i=1}^k$ such that
the restriction $\bar f|_{C_{-i}}$, $i=1,...,k$, is non-simple
which factors through a non-trivial ramified covering of degree $\delta_i\glt2$
of a simple map $f_i:\pb^1\to X^-$.
Assume $(f_i)_*[\pb^1]=d_i$, $i=1,...,k$, and $\bar f_*[C_{-j}]=d_j$, $j=k+1,...,m$.
then $\sum_{i=1}^k\delta_id_i+\sum_{j=k+1}^{m} d_j=p^!d-b[E]$.
\begin{eqnarray*}
& & c_1(\bar{X}^-)\cdot(\sum_{i=1}^kd_i)-k+c_1(\bar{X}^-)\cdot(\sum_{j=k+1}^md_j)-(m-k)\\
& & \\
&  & \geq    c_1(X)\cdot d-1 \\
&  & =  c_1(\bar{X}^-)\cdot(\sum_{i=1}^k\delta_id_i+\sum_{j=k+1}^md_j)+b-1.
\end{eqnarray*}
Therefore,
\begin{equation}\label{3-6}
\sum_{i=1}^k(1-\delta_i)c_1(\bar{X}^-)\cdot d_i
\glt m+b-1.
\end{equation}
Since $c_1(\bar X^-)\cdot d_i \glt 0$, $b\glt0$, $\delta_i\glt2$,
so (\ref{3-6}) holds only when $m=1, k=0 $ and $b=0$.
This implies that $C_{-}$ is irreducible and $\bar f|_{C_{-}}$ is simple.
$ b=0$  also implies  $\bar f_*[C_+]=0$.
  Therefore, $C_+=\emptyset$.

The previous argument implies that  $\bar f|_{C_-}$
is an element of $\cl^{0,0}(p^!d,\undl x(0)\cap \bar X^-,J_0)$. Moreover,
the finiteness of $\cl^{0,0}(p^!d,\undl x(0)\cap \bar X^-,J_0)$
implies that $\cl(d,\undl x(0), J_0)$ is finite.

{\bf Case II :} $\undl x(0)\cap\bar X^+=\{p\}$.

   In this case, the fact that the image of $\bar f(C_+)$
  has to passes $\{p\}$  implies $a=b\glt1$.
  $\bar f(C_{-})$ passes through all the $c_1(X)\cdot d-2$ points in $\undl x(0)\cap X^-$
  and realizes the class $p^!d-b[E]$ in $H_2(\bar X^-;\zb)$.
  Similar to case I, we know that $C_{-}$ is irreducible. Next, we prove that $\bar f|_{C_{-}}$ is simple.
  For this, we assume that  $\bar f|_{C_{-}}$ is non-simple. Then $\bar f|_{C_{-}}$
  factors through a non-trivial ramified covering of degree $\delta\glt2$
  of a simple map $f_0:\pb^1\to\bar X^-$, and  $(f_0)_*[\pb^1]=\frac{1}{\delta}(p^!d-b[E])$. Therefore
  \begin{equation*}
  \frac{1}{\delta}c_1(\bar X^-)\cdot(p^!d-b[E])-1 \glt c_1(X)\cdot d-2.
\end{equation*}
So we have
\begin{equation}\label{3-2}
  c_1(X)\cdot d+\delta-\delta c_1(X)\cdot d \glt b.
\end{equation}
Since $\delta\glt2$, $c_1(X)\cdot d\glt2$, (\ref{3-2}) implies $b\leq 0$.
  This is in contradiction with $b\geq 1$. Therefore, $\bar f|_{C_{-}}$ is simple.

From
\begin{eqnarray*}
  c_1(\bar X^-)\cdot (p^!d-b[E])-1 & = & c_1(X)\cdot d-1-b\\
  & \glt & c_1(X)\cdot d-2.
\end{eqnarray*}
we obtain $b\leq 1$. So we have $b=1$, and
  $\bar f_*[C_+]=[H]$.

The previous argument implies that $\bar f|_{C_-}$
  is an element of $\cl^{0,\delta_1}(p^!d-[E],\undl x(0)\cap\bar X^-,J_0)$.
  The finiteness of $\cl^{0,\delta_1}(p^!d-[E],\undl x(0)\cap\bar X^-,J_0)$
  implies that $\cl(d,\undl x(0), J_0)$ is finite.

The number of elements of $\cl(d,\undl x(\lambda),J_\lambda)$ converging to $\bar f$ as $\lambda$
goes to $0$ follows from \cite{ip2004}.
Let's review the behavior of an elements
$f_\lambda:C_\lambda\to\zl_\lambda$ of $\cl(d,\undl x(\lambda),J_\lambda)$
converging to $\bar f$ close to the smoothing of the intersection point $p$ of $C_-$ and
$C_+$. In local coordinates $(\lambda,x,y)$ at $\bar f(p)$, the manifold $\zl$ is given
by the equation $xy=\lambda$. Locally,
$$
\bar{X}^+=\{\lambda=0\text{ and } y=0\}, \,\,\,\, \bar{X}^-=\{\lambda=0\text{ and }x=0\}.
$$
Since the order of intersection of $\bar f_{C_+}$ and $V$ at $\bar f(p)$ is $1$,
the maps $\bar f_{C_+}$ and $\bar f_{C_-}$ have expansions
$$
x(z)=mz+o(z)\,\,\, \mbox{  and  }\,\,\, y(w)=nw+o(w),
$$
where $z$ and $w$ are local coordinates at $p$ of $C_+$ and $C_-$ respectively.

For $0<|\lambda|\ll1$, there exists a solution $\mu(\lambda)\in\cb^*$ of
$$
\mu(\lambda)=\frac{\lambda}{mn}
$$
such that the smoothing of $\bar C$ at $p$ is locally given by $zw=\mu(\lambda)$, and the
map $f_\lambda$ is approximated by the map
$$
\{zw=\mu(\lambda)\}\subset\cb^2\mapsto(\lambda,mz,nw)
$$
close to the smoothing of $p$. Furthermore, such maps $f_\lambda\in\cl(d,\undl x(\lambda),J_\lambda)$
converging to $\bar f$ are in one to one correspondence with the choice of such $\mu(\lambda)$ for each
point of $C_+\cap C_-$.

\end{proof}

Applying Proposition $\ref{prop:sum11}$, we can get a comparison theorem of the Welschinger invariants.
Let $\rb\cl^{\alpha^r+\alpha^c,\beta^r+\beta^c}(d,\undl x,J)$ be the set of real rational
curves in $\cl^{\alpha,\beta}(d,\undl x,J)$, $\alpha=\alpha^r+\alpha^c$
and $\beta=\beta^r+\beta^c$, such that the $\alpha$ (or $\beta$) "points"  consists of
$\alpha^r$ (or $\beta^r$) real "points" and $\frac{1}{2}\alpha^c$ (or $\frac{1}{2}\beta^c$)
pairs of $\tau$-conjugated "points".

\begin{prop}\label{prop:welrep11}
Let $X_\rb$ be a compact real symplectic $4$-manifold,
$d\in H_2(X;\zb)$ such that $c_1(X)\cdot d>0$ and $\tau_*d=-d$.
Denote by $p: X_{1,0}\to X$ the projection of the real symplectic blow-up of $X$ at $x\in\rb X$.
Let $\undl x(\lambda)$, $\bar{X}^+$, $\bar{X}^-$, $J_0$ be as before.
Then
\begin{itemize}
  \item if $\undl x(0)\cap \bar{X}^+=\{p\}$, $\undl x(0)\cap \bar{X}^-\neq\emptyset$,
  \begin{equation}\label{eq:rcomp1}
  W_{X_\rb}(d,s)=\sum_{C_-\in\rb\cl^{0,\delta_1^r}(p^!d-[E],\undl x(0)\cap \bar{X}^-,J_0)}(-1)^{m_{X_{1,0}}(C_-)},
\end{equation}
\item if $\undl x(0)\cap \bar{X}^+=\emptyset$,
\begin{equation}\label{3-4}
  W_{X_\rb}(d,s)=\sum_{C_-\in\rb\cl^{0,0}(p^!d,\undl x(0)\cap \bar{X}^-,J_0)}(-1)^{m_{X_{1,0}}(C_-)}.
\end{equation}
\end{itemize}
where $E$ is the exceptional divisor.
\end{prop}

\begin{proof}
Equip the small disc $\Delta$ with the standard complex conjugation.
From subsection $\ref{subsec:realsympcut}$,
we know one can equip the symplectic sum $\pi:\zl\to\Delta$ with a real structure $\tau_{\zl}$
which is induced by the real structures $\tau_-$, $\tau_+$
on the real symplectic cuts $\bar{X}^-$ and $\bar{X}^+$
such that the map $\pi:\zl\to\Delta$ is real.
Choose a set of real sections $\undl x:\Delta\to\zl$ .
Let $\bar f:\bar C\to X_\sharp$ be a real element in $\rb\cl(d,\undl x(0),J_0)$.

For the case $\undl x(0)\cap \bar{X}^+=\{p\}$ and $\undl x(0)\cap \bar{X}^-\neq\emptyset$,
from Proposition $\ref{prop:sum11}$ and Lemma $\ref{lem:p2set}$
we know $\bar f_*[C_+]=[H]$ and $\bar f|_{C_+}$ is an embedded simple curve.
$\bar f(C_+)$ has no self-intersection point, so $\bar f(C_+)$ has no node.
And there is only one possibility for $\bar f|_{C_+}$ to recover a real curve $\bar f(\bar C)$
when $\bar f|_{C_-}$ is fixed. In other words,
the number of real curves $\bar f\in\rb\cl(d,\undl x(0),J_0)$ is equal to
the number of the real curves $\bar f|_{C_-}\in\rb\cl^{0,\delta_1}(p^!d-[E],\undl x(0)\cap \bar{X}^-,J_0)$.
Therefore we have
\begin{eqnarray}\label{3-3}
m_{X_\sharp}(\bar f(\bar C)) & = & m_{\bar{X}^-}(\bar f|_{\bar C_-})+m_{\bar{X}^+}(\bar f|_{C_+})\nonumber\\
  & = & m_{\bar{X}^-}(\bar f|_{\bar C_-}).
\end{eqnarray}
By Proposition $\ref{prop:sum11}$, an element $\bar f$ of $\cl(d,\undl x(0),J_0)$ is the limit
of a unique element of $\cl(d,\undl x(\lambda),J_\lambda)$, so the latter has to be real when
$\bar f$ is real and $\lambda\in\cb^*$ is small. The description at the end of the proof of
Proposition $\ref{prop:sum11}$ of the local deformation of $\bar f$ shows that no node appears
in a neighborhood of $V\cap\bar f(\bar C)$ when deforming $\bar f$. Combining with (\ref{3-3}), this implies (\ref{eq:rcomp1}).

For the case $\undl x(0)\cap \bar{X}^+=\emptyset$,
we know $\bar f_*[C_+]=0$ from Proposition $\ref{prop:sum11}$.
The real curve $\bar f(\bar C)$ is determined by the part $\bar f|_{C_-}$.
In other words,
the number of real curves $\bar f\in\rb\cl(d,\undl x(0),J_0)$ is equal to
the number of the real curves $\bar f|_{C_-}\in\rb\cl^{0,0}(p^!d,\undl x(0)\cap \bar{X}^-,J_0)$.
Then we have
$$
m_{X_\sharp}(\bar f(\bar C))=m_{X^-}(\bar f|_{\bar C_-})
$$

The rest of the Proposition can be proved similar to the previous case.
\end{proof}

\begin{rem}
Proposition $\ref{prop:welrep11}$ tells us that the sum
$$
\sum_{C_-\in\rb\cl^{0,\delta_1^r}(p^!d-[E],\undl x(0)\cap\bar X^-,J_0)}(-1)^{m_{X_{1,0}}(C_-)}
$$
on the right side of formula $(\ref{eq:rcomp1})$
does not dependent on $\undl x(0)\cap\bar X^-$ and $J_0$.
It can be seen as a particular case of relative Welschinger invariants.
See \cite{brugalle2016} and \cite{iks2014rw} for more about relative Welschinger invariants.
\end{rem}

Now we take real symplectic cut along the exceptional divisor $E$
of the real symplectic blow-up manifold $\tilde X=X_{1,0}$.
We can get two real symplectic cuts:

$$
\overline{X}_{1,0}^+\cong\pb( N_{E|X_{1,0}}\oplus\ol_E)\cong\pb_E(\ol(-1)\oplus\ol)
\,\,\,\mbox{ and } \,\,\, \overline{ X}_{1,0}^-\cong X_{1,0},
$$
which contain a common real symplectic submanifold $V$.
In $\overline{ X}_{1,0}^+ $, $V\cong E_\infty$  is the infinity section of $\pb_E(\ol(-1)\oplus\ol)\to E$.
In $\overline{ X}_{1,0}^-\cong X_{1,0}$, $V\cong E$ is the exceptional divisor.

Let $\tilde\zl$ be the real symplectic sum of $\overline{ X}_{1,0}^+$ and $\overline{X}_{1,0}^-$ along $V$.
Let $p^!d-[E]\in H_2(\tilde\zl_\lambda;\zb)$, where $d\in H_2(X;\zb)$.
Choose $\tilde{\undl x}_1(\lambda)$
a set of $c_1(X)\cdot d-1$ real symplectic sections $\Delta\to\tilde\zl$ such that
$\tilde{\undl x}_1(0)\cap V=\emptyset$, $\tilde{\undl x}_1(0)\cap \overline{X}_{1,0}^+=\emptyset$.
Choose $\tilde{\undl x}_2(\lambda)$
a set of $c_1(X)\cdot d-2$ real symplectic sections $\Delta\to\tilde\zl$ such that
$\tilde{\undl x}_2(0)\cap V=\emptyset$, $\tilde{\undl x}_2(0)\cap\overline{X}^+_{1,0}=\emptyset$.
Choose a generic almost complex structure $\tilde J$ on $\tilde\zl$ as above.

Let $\tilde X_\sharp=\overline{X}_{1,0}^+\cup_V\overline{X}_{1,0}^-$.
Define $\cl(p^!d,\tilde{\undl x}_1(0),\tilde J_0)$,
$\cl(p^!d-[E],\tilde{\undl x}_2(0),\tilde J_0)$ to be the set
$\{\bar f:\bar C\to\tilde X_\sharp\}$ of limits, as stable maps,
of maps in $\cl(p^!d,\tilde{\undl x}_1(\lambda),\tilde J_\lambda)$,
$\cl(p^!d-[E],\tilde{\undl x}_2(\lambda),\tilde J_\lambda)$ as $\lambda$ goes to $0$, respectively.
Given an element
$\bar f:\bar C\to\tilde X_\sharp$ of $\cl(p^!d,\tilde{\undl x}_2(0),\tilde J_0)$ or
$\cl(p^!d-[E],\tilde{\undl x}_2(0),\tilde J_0)$, Denote by $C_*$, $*=+,-$, the union of the irreducible components
of $\bar C$ mapped to $\overline{X}_{1,0}^*$.

\begin{prop}\label{prop:sum12}
Under the assumption above, we have the following:

  (1) For a generic $\tilde J_0$, the set $\cl(p^!d,\tilde{\undl x}_1(0),\tilde J_0)$ is finite,
  and only depends on $\tilde{\undl x}_1(0)$ and $\tilde J_0$.
  Given an element $\bar f:\bar C\to\tilde X_\sharp$  of
  $\cl(p^!d,\tilde{\undl x}_1(0),\tilde J_0)$,
  the restriction of $\bar f$ to any component of $\bar C$ is a simple map,
  and no irreducible component of $\bar C$ is entirely mapped into $V$.
  Moreover, the curve $C_-$ is irreducible and $\bar f|_{C_-}$ is an element of
  $\cl^{0,0}(p^!d,\tilde{\undl x}_1(0)\cap\overline{X}_{1,0}^-,\tilde J_0)$.
  The map $\bar f$ is the limit of a unique element of
  $\cl(p^!d,\tilde{\undl x}_1(\lambda),\tilde J_\lambda)$ as $\lambda$ goes to $0$.

 (2) For a generic $\tilde J_0$, the set $\cl(p^!d-[E],\tilde{\undl x}_2(0),\tilde J_0)$ is finite,
  and only depends on $\tilde{\undl x}_2(0)$ and $\tilde J_0$.
  Given an element $\bar f:\bar C\to\tilde X_\sharp$  of
  $\cl(p^!d-[E],\tilde{\undl x}_2(0),\tilde J_0)$,
  the restriction of $\bar f$ to any component of $\bar C$ is a simple map,
  and no irreducible component of $\bar C$ is entirely mapped into $V$.
  Moreover, the curve $C_-$ is irreducible and $\bar f|_{C_-}$ is an element of
  $\cl^{0,\delta_1}(p^!d-[E],\tilde{\undl x}_2(0)\cap\overline{X}_{1,0}^-,\tilde J_0)$.
  The map $\bar f$ is the limit of a unique element of
  $\cl(p^!d-[E],\tilde{\undl x}_2(\lambda),\tilde J_\lambda)$ as $\lambda$ goes to $0$.

\end{prop}

\begin{proof}
The fact that no component of $\bar C$ is entirely mapped into $V$ follows from
\cite[Example 11.4 and Lemma 14.6]{ip2004}, also see \cite{bp2014}.

(1)  Suppose $\bar f_*[C_+]=a[F]+b[E_\infty]$, $\bar f_*[C_-]=p^!d-k[E]$, $k\geq 0$
  where $F$ is a fiber of $\pb_E(\ol(-1)\oplus\ol)\to E$ with $F\cdot[E_0]=1$ and $F\cdot[E_\infty]=1$. Then
  \begin{align*}
  a+b=(a[F]+b[E_\infty])\cdot[E_\infty]&=(p^!d-k[E])\cdot[E]=k,\\
  a=(a[F]+b[E_\infty])\cdot[E_0]&=(p^!d)\cdot[E]=0.
  \end{align*}
  In $\overline{X}_{1,0}^-$, we know $\bar f|_{C_-}$ passes through
  $$
  |\tilde{\undl x}_1(0)|=c_1(X)\cdot d-1=c_1(\overline{X}_{1,0}^-)\cdot (p^!d)-1
  $$
  distinct points in $\overline{X}_{1,0}^-$.
 The same argument as in the proof of Proposition $\ref{prop:sum11}$ shows that $C_{-}$ is irreducible.

  Next, we prove that $\bar f|_{C_{-}}$ is simple. For this, we assume that $\bar f|_{C_{-}}$ is non-simple. Then $\bar f|_{C_{-}}$
  factors through a non-trivial ramified covering of degree $\delta\glt2$
  of a simple map $f_0:\pb^1\to \overline{X}_{1,0}^-$, and $(f_0)_*[\pb^1]=\frac{1}{\delta}(p^!d-k[E])$. Therefore,
  \begin{eqnarray}\label{3-5}
  \frac{1}{\delta}c_1(\overline{X}_{1,0}^-)\cdot(p^!d-k[E])-1 & \glt & c_1(\overline{X}_{1,0}^-)\cdot (p^!d)-1\nonumber\\
  (1-\delta) c_1(\overline{X}_{1,0}^-)\cdot (p^!d) & \glt & k.
  \end{eqnarray}
Since $c_1(X)\cdot d\glt1$, $\delta\glt2$, $k\glt0$, so (\ref{3-5}) is impossible. Therefore,
 $\bar f|_{C_{-}}$ can only be simple.

  On the other hand, we have
  \begin{eqnarray*}
  c_1(\overline{X}_{1,0}^-)\cdot(p^!d-k[E])-1 & = & c_1(\overline{X}_{1,0}^-)\cdot p^!d-k-1\\
  & \glt &  c_1(\overline{X}_{1,0}^-)\cdot (p^!d)-1.
  \end{eqnarray*}
This implies $k=0$ and $b=0$. Therefore, $C_{+}=\emptyset$ and $\bar f|_{C_-}$ is an element of
  $\cl^{0,0}(p^!d,\tilde{\undl x}_1(0)\cap\overline{X}_{1,0}^-,\tilde J_0)$
  which is also a simple map.

(2) Suppose $\bar f_*[C_+]=a[F]+b[E_\infty]$,  $\bar f_*[C_-]=p^!d-k[E]$, $k\glt1$,
  where $F$ is a fiber of $\pb_E(\ol(-1)\oplus\ol)\to E$ with $F\cdot[E_0]=1$ and $F\cdot[E_\infty]=1$. Then
\begin{eqnarray*}
  a+b & = & (a[F]+b[E_\infty])\cdot[E_\infty]=(p^!d-k[E])\cdot[E]=k,\\
   a & = & (a[F]+b[E_\infty])\cdot[E_0]=(p^!d-[E])\cdot[E]=1.
\end{eqnarray*}

  In $\overline{X}_{1,0}^-$, we know that $\bar f|_{C_-}$ passes through
  $$
  c_1(X)\cdot d-2=c_1(\overline{X}_{1,0}^-)\cdot (p^!d-[E])-1
  $$
  distinct points.
  The same argument as in the proof of Proposition $\ref{prop:sum11}$ shows that $C_{-}$ is irreducible.

  By a similar analysis of the dimension condition as in the proof of case (1), we obtain
  $\bar f_*[C_+]=[F]$,
  $C_+$ must have exactly $1$ component and the image of it realize the class $[F]$.
  $\bar f|_{C_-}$ is an element of
  $\cl^{0,\delta_1}(p^!d-[E],\tilde{\undl x}_2(0)\cap\overline{X}_{1,0}^-,\tilde J_0)$
  which is a simple map.

  The proof of other parts is the same as that in the proof of Proposition $\ref{prop:sum11}$.
We omit it here.
\end{proof}

\begin{prop}\label{prop:welrep12}
Let $X_\rb$ be a compact real symplectic $4$-manifold,
$d\in H_2(X;\zb)$ such that $c_1(X)\cdot d-1=r+2s>0$ and $\tau_* d =-d$.
Denote by $p: X_{1,0}\to X$ the projection of the real symplectic blow-up of $X$ at $x\in\rb X$.
Let $\tilde{\undl x}_1$, $\tilde{\undl x}_2$, $\tilde J_0$, $\overline{X}_{1,0}^+$, $\overline{X}_{1,0}^-$
be as before.
Then
\begin{eqnarray*}
W_{X_{1,0}}(p^!(d),s)
& = & \sum_{C_-\in\rb\cl^{0,0}(p^!d,\tilde{\undl x}_1(0)\cap\overline{X}_{1,0}^-,\tilde J_0)}(-1)^{m_{X_{1,0}}(C_-)},\\
& & \\
W_{X_{1,0}}(p^!(d)-[E],s)
& = & \sum_{C_-\in\rb\cl^{0,\delta_1^r}(p^!d-[E],\tilde{\undl x}_2(0)\cap\overline{X}_{1,0}^-,\tilde J_0)}(-1)^{m_{X_{1,0}}(C_-)},
\end{eqnarray*}
where $E$ is the exceptional divisor.
\end{prop}

\begin{rem}
Similar
to Proposition $\ref{prop:welrep11}$, by Proposition $\ref{prop:sum12}$,  one can prove Proposition $\ref{prop:welrep12}$.
Moreover, Proposition $\ref{prop:welrep11}$ and Proposition $\ref{prop:welrep12}$ imply Theorem $\ref{thm:bpr}$.
\end{rem}

\subsection{Blow-up formula at a conjugated pair}\label{subsec:bpf2}

Let $(X,\omega)$ be a compact  connected  real symplectic $4$-manifold,
and let $V_1$,$V_2\subset X$ be two disjoint embedded symplectic curves in $X$.
Let $d\in H_2(X;\zb)$ and $\alpha^1$, $\alpha^2$, $\beta^1$, $\beta^2\in\zb_{\glt0}^\infty$ such that
$$
I\alpha^1+I\beta^1=d\cdot[V_1],~I\alpha^2+I\beta^2=d\cdot[V_2].
$$
Choose a configuration $\undl x=\undl x^\circ\sqcup\undl x_{V_1}\sqcup\undl x_{V_2}$ of points in $X$,
with $\undl x^\circ$ a configuration of $c_1(X)\cdot d-1-d\cdot([V_1]+[V_2])+|\beta_1|+|\beta_2|$
points in $X\setminus (V_1\cup V_2)$, $\undl x_{V_1}=\{p_{i,j}\}_{0<j\llt\alpha^1_i,i\glt1}$ a configuration
of $|\alpha^1|$ points in $V_1$, and $\undl x_{V_2}=\{q_{i,j}\}_{0<j\llt\alpha^2_i,i\glt1}$ a configuration
of $|\alpha^2|$ points in $V_2$. Given  an $\omega$-tamed almost complex structure $J$ on $X$  such that $V_1$ and $V_2$ are $J$-holomorphic,
 denote by $\cl^{\alpha^1,\beta^1,\alpha^2,\beta^2}(d,\undl x,J)$ the set
of rational $J$-holomorphic maps $f:\cb P^1\to X$ such that

$\bullet$ $f_*[\cb P^1]=d$;

$\bullet$ $\undl x\subset f(\cb P^1)$;

$\bullet$ $V_1\cup V_2$ does not contain $f(\cb P^1)$;

$\bullet$ $f(\cb P^1)$ has order of contact $i$ with $V_1$ at each points $p_{i,j}$
          and has order of contact $i$ with $V_2$ at each points $q_{i,j}$;

$\bullet$ $f(\cb P^1)$ has order of contact $i$ with $V_1$ at exactly $\beta^1_i$
          distinct points on $V_1\setminus\undl x_{V_1}$ and
          has order of contact $i$ with $V_2$ at exactly $\beta^2_i$
          distinct points on $V_2\setminus\undl x_{V_2}$.

Note that $\cl^{\alpha^1,\beta^1,\alpha^2,\beta^2}(d,\undl x,J)$
may contain components of positive dimension corresponding to non-simple maps. But for the generic $J$,
the set of simple maps in $\cl^{\alpha^1,\beta^1,\alpha^2,\beta^2}(d,\undl x,J)$ is $0$-dimensional.

\begin{lem}\label{lem:finite2}
Let $(X,\omega)$ be a compact connected   real symplectic $4$-manifold. Suppose that $V_1$ and $V_2$ are
two embedded real symplectic spheres in $X$ with $V_1\cdot V_2=0$, $[V_i]^2=-1$, $i=1,2$,
and assume $|\beta^i|=d \cdot[V_i]$, $i=1,2$.
Then for a generic choice of $J$, the set
$\cl^{\alpha^1,\beta^1,\alpha^2,\beta^2}(d,\undl x,J)$ is finite and contains
only simple maps that are all immersions.
\end{lem}

\begin{proof}
Suppose that $\cl^{\alpha^1,\beta^1,\alpha^2,\beta^2}(d,\undl x,J)$
contains a non-simple map which factors through a non-trivial ramified covering
of degree $\delta$ of a simple map $f_0:\cb P^1\to X$.
Let $d_0$ denotes the homology class $(f_0)_*[\cb P^1]$.
Since $f_0(\cb P^1)$ passes through $c_1(X)\cdot d -1 = \delta c_1(X)\cdot d_0-1$ points,
we have
$$
c_1(X)\cdot d_0-1\glt\delta c_1(X)\cdot d_0-1\glt0,
$$
which is impossible.

Suppose that $\cl^{\alpha^1,\beta^1,\alpha^2,\beta^2}(d,\undl x,J)$
contains infinitely many simple maps. By Gromov compactness Theorem,
There exists a sequence $(f_n)_{n\glt0}$ of simple maps in
$\cl^{\alpha^1,\beta^1,\alpha^2,\beta^2}(d,\undl x,J)$ which converges
to some $J$-holomorphic map $\bar f:\bar C\to X$.
By genericity of $J$, the set of simple maps in
$\cl^{\alpha^1,\beta^1,\alpha^2,\beta^2}(d,\undl x,J)$ is discrete.
Hence either $\bar C$ is reducible, or $\bar f$ is non-simple.
Let $\bar C_1$,...,$\bar C_m$, $\bar C_1^1$,...,$\bar C_{m_1}^1$,
$\bar C_1^2$,...,$\bar C_{m_2}^2$
be the irreducible components of $\bar C$, labeled in such a way that

$\bullet$ $\bar f(\bar C_i)\nsubseteq V_1\cup V_2$;

$\bullet$ $\bar f(\bar C_i^j)\subset V_j$ and $(\bar f)_*[\bar C_i^j]=k_i^j[V_j]$, $j=1,2$.\\
Define $k^j=\sum^{m_j}_{i=1}k_i^j$, $j=1,2$.
The restriction of $\bar f$ to $\cup_{i=1}^m\bar C_i$ is subject to
$c_1(X)\cdot d-1-d\cdot([V_1]+[V_2])+|\beta^1|+|\beta^2|$ points constrains,
so we have
$$
c_1(X)\cdot(d-k^1[V_1]-k^2[V_2])-m\glt c_1(X)\cdot d-1-d\cdot([V_1]+[V_2])+|\beta^1|+|\beta^2|.
$$
Since both $V_1$ and $V_2$ are  embedded real symplectic spheres, from adjunction formula,
we can get $c_1(X)\cdot[V_i]=1$, $i=1,2$ .
Hence we obtain
$$
c_1(X)\cdot d-k^1-k^2-m\glt c_1(X)\cdot d-1-d\cdot([V_1]+[V_2])+|\beta^1|+|\beta^2|.
$$
From the assumption of Lemma, we have
$$
0\glt  m+k^1+k^2-1 .
$$
Therefore, we have $m=1$, $k^1=k^2=0$.
The map $\bar f$ has to be a non-simple map which factors through a non-trivial
ramified covering of a simple map $f_0:\cb P^1\to X$.
But $f_0$ is subject to more point constraints, which provides a contradiction.
\end{proof}

Let $X_\rb$ be a compact real symplectic $4$-manifold, and suppose
$y_1$, $y_2\in X\setminus\rb X$ is a $\tau$-conjugated pair.
Denote by $p:{X}_{0,1}\to X$ the projection of the real symplectic blow-up of $X$ at $y_1$, $y_2$.
Perform real symplectic cut of $X$ at the $\tau$-conjugated pair $y_1$, $y_2$. We get
$$
\bar{X}^+=\bar{X}^{+1}\sqcup \bar{X}^{+2}\cong\pb^2\sqcup\pb^2,\,\,\,\, \bar{X}^-\cong X_{0,1}.
$$
Both $\bar{X}^+$ and $\bar{X}^-$ contain two common real symplectic submanifolds $V_1$, $V_2$ of real codimension $2$.
Let $V=V_1\sqcup V_2$.
In $\bar{X}^+$, $V_1\cong H_1$, $V_2\cong H_2$, where $H_i$ is the hyperplane of $\bar{X}^{+i}$, $i=1,2$.
$V_1\cong E_1$ and $V_2\cong E_2$  are the associated exceptional divisors in $\bar{X}^-$  at $y_i$, $i=1,2$, respectively.

Let $\zl$ be the real symplectic sum of the two real symplectic manifolds $\bar{X}^+$ and $\bar{X}^-$ along $V$,
and $d\in H_2(\zl_\lambda;\zb)$. Denote ${\undl x}(\lambda)$, $J$,
$X_\sharp=\bar{X}^+\cup_V \bar{X}^-$, $\cl(d,{\undl x}(0),J_0)$, $C_{+i}$, $i=1,2$,
as in subsection $\ref{subsec:bpform}$. Similar to the proof of Proposition $\ref{prop:sum11}$, we can prove

\begin{prop}\label{prop:sum21}
Assume $\undl x(0)\cap \bar{X}^{+i}$ contains at most one point,
$\undl x(0)\cap \bar{X}^-\neq\emptyset$ if $\undl x(0)\cap \bar{X}^+\neq\emptyset$.
Let ${\undl x}(0)$, $d$, $J$ be given above.
Then for a generic $J_0$, the set $\cl(d,{\undl x}(0),J_0)$ is finite,
and only depends on ${\undl x}(0)$ and $J_0$.
Given an element $\bar f:\bar C\to X_\sharp$  of $\cl(d, \undl x(0),J_0)$,
the restriction of $\bar f$ to any component of $\bar C$ is a simple map,
and no irreducible component of $\bar C$ is entirely mapped into $V$.
Moreover, the followings are true:

(1) If ${\undl x}(0)\cap \bar{X}^{+i}=\{p_i\}$, $i=1,2$,   the curve $C_-$ is irreducible and $\bar f|_{C_-}$ is an element of
$\cl^{0,\delta_1,0,\delta_1}(p^!d-[E_1]-[E_2],{\undl x}(0)\cap \bar{X}^-,J_0)$.  The curves $C_{+i}$, $i=1,2$, are irreducible
and the image of $C_{+i}$ represses $[H_i]$ and passes $\{p_i\}$, respectively. The map $\bar f$ is the limit of a unique element
of $\cl(d, {\undl x}(\lambda),J_\lambda)$ as $\lambda$ goes to $0$.

(2) If ${\undl x}(0)\cap \bar{X}^{+i}=\emptyset$, $i=1,2$, then $C_{+i}=\emptyset$, the curve $C_-$ is irreducible and
$\bar f|_{C_-}$ is an element of  $\cl^{0,0,0,0}(p^!d,{\undl x}(0)\cap \bar{X}^-,J_0)$. The map $\bar f$ is the limit of a
unique element of   $\cl(d, {\undl x}(\lambda),J_\lambda)$ as $\lambda$ goes to $0$.

\end{prop}

Base on Proposition $\ref{prop:sum21}$, similar
to Proposition $\ref{prop:welrep11}$, we can prove

\begin{prop}\label{prop:welrep21}
Let $X_\rb$ be a compact real symplectic $4$-manifold,
$d\in H_2(X;\zb)$ such that $c_1(X)\cdot d>0$ and $\tau_*d=-d$.
Suppose $y_1$, $y_2\in (X\setminus\rb X)$ is a $\tau$-conjugated pair.
Denote by $p:X_{0,1}\to X$ the projection of the real symplectic blow-up of $X$ at $y_1$, $y_2$.
Then

(1) If $\undl x(0)\cap \bar{X}^{+i}=\emptyset$, $i=1,2$, then
  $$
  W_{X_\rb}(d,s)=\sum_{C_-\in\rb\cl^{0,0,0,0}(p^!d,{\undl x}(0)\cap \bar{X}^-,J_0)}(-1)^{m_{X_{0,1}}(C_-)}.
  $$

(2) If $\undl x(0)\cap \bar{X}^{+i}=\{p_i\}$, $i=1,2$, $\undl x(0)\cap \bar{X}^-\neq\emptyset$, then
$$
  W_{X_\rb}(d,s)=\sum_{C_-\in\rb\cl^{0,\delta_1^c,0,\delta_1^c}
(p^!d-[E_1]-[E_2],{\undl x}(0)\cap \bar{X}^-,J_0)}(-1)^{m_{X_{0,1}}(C_-)},
$$
where $E_1$, $E_2$ are the exceptional divisors.

\end{prop}

Next, perform the real symplectic cut of $X_{0,1}$ along $E_1$, $E_2$,
where $E_i$ is the exceptional divisor of the blow-up at $y_i$, $i=1,2$, respectively.
We obtain two real symplectic manifolds $\overline{\tilde X}^+=\overline{\tilde X}^{+1}\sqcup \overline{\tilde X}^{+2}$ and $\overline{\tilde X}^-$ as follow:
$$
\overline{\tilde X}^+\cong\pb_{E_1}(\ol(-1)\oplus\ol)\sqcup\pb_{E_2}(\ol(-1)\oplus\ol), \,\,\,\,
\overline{\tilde X}^-\cong X_{0,1}.
$$
Both $\overline{\tilde X}^+$ and $\overline{\tilde X}^-$ contain a common real symplectic submanifold
$V_1$, $V_2$ of real codimension $2$ respectively.
Let $V=V_1\sqcup V_2$.
In $\overline{\tilde X}^+$, $V_1\cong E^1_{\infty}$,~$V_2\cong E^2_\infty$,
where $E^i_\infty$ is the infinity section of $\pb_{E_i}(\ol(-1)\oplus\ol)\to E_i$, respectively.
$V_1\cong E_1$, $V_2\cong E_2$ are the exceptional divisors in $\overline{\tilde X}^-$.

Let $\tilde\zl$ be the symplectic sum of the two real symplectic manifolds $\overline{\tilde X}^+$ and $\overline{\tilde X}^-$ along $V$.
Let $p^!d-[E_1]-[E_2]\in H_2(\tilde\zl_\lambda;\zb)$, where $d\in H_2(X;\zb)$.
Choose $\tilde{\undl x}_1(\lambda)$
a set of $c_1(X)\cdot d-3$ real symplectic sections $\Delta\to\tilde\zl$ such that
$\tilde{\undl x}_1(0)\cap V=\emptyset$, $\tilde{\undl x}_1(0)\cap\overline{\tilde X}^+=\emptyset$.
Choose $\tilde{\undl x}_2(\lambda)$
a set of $c_1(X)\cdot d-1$ real symplectic sections $\Delta\to\tilde\zl$ such that
$\tilde{\undl x}_2(0)\cap V=\emptyset$, $\tilde{\undl x}_2(0)\cap\overline{\tilde X}^+=\emptyset$.
Choose a generic almost complex structure $\tilde J$ on $\tilde \zl$ as above.

Denote $\tilde X_\sharp=\overline{\tilde X}^+\cup_V \overline{\tilde X}^-$,
$\cl(p^!d-[E_1]-[E_2],\tilde{\undl x}_1(0),\tilde J_0)$,
$\cl(p^!d,\tilde{\undl x}_2(0),\tilde J_0)$, $C_{+i}$,
$i=1,2$ and $C_-$ as in subsection $\ref{subsec:bpform}$.
The same argument as in the proof of Proposition $\ref{prop:sum11}$ and Proposition $\ref{prop:welrep11}$
shows that Proposition $\ref{prop:sum22}$ and Proposition $\ref{prop:welrep22}$ hold.

\begin{prop}\label{prop:sum22}
Let $\tilde{\undl x}(0)$, $p^!d-[E_1]-[E_2]$, $\tilde J$ be given above.
Then we have

 (1)  For a generic $\tilde J_0$,
  the set $\cl(p^!d-[E_1]-[E_2],\tilde{\undl x}_1(0),\tilde J_0)$ is finite,
  and only depends on $\tilde{\undl x}_1(0)$ and $\tilde J_0$.
  Given an element $\bar f:\bar C\to \tilde X_\sharp$  of
  $\cl(p^!d-[E_1]-[E_2],\tilde{\undl x}_1(0),\tilde J_0)$,
  the restriction of $\bar f$ to any component of $\bar C$ is a simple map,
  and no irreducible component of $\bar C$ is entirely mapped into $V$.
  Moreover, the curve $\bar f|_{C_-}$ is irreducible and $\bar f|_{C_-}$ is an element of
  $\cl^{0,\delta_1,0,\delta_1}(p^!d-[E_1]-[E_2],\tilde{\undl x}_1(0)\cap\overline{\tilde X}^-,\tilde J_0)$.
  $C_{+i}$, $i=1,2$, are irreducible and the image of $C_{+i}$ under $\bar{f}$ represents the fiber class $[F_i]$ of $\pb_{E_i}(\ol(-1)\oplus\ol)\to E_i$, respectively.
  The map $\bar f$ is the limit of a unique element of
  $\cl(p^!d-[E_1]-[E_2], \tilde{\undl x}_1(\lambda),\tilde J_\lambda)$ as $\lambda$ goes to $0$.

(2) For a generic $\tilde J_0$, the set $\cl(p^!d,\tilde{\undl x}_2(0),\tilde J_0)$ is finite,
  and only depends on $\tilde{\undl x}_2(0)$ and $\tilde J_0$.
  Given an element $\bar f:\bar C\to \tilde X_\sharp$  of $\cl(p^!d,\tilde{\undl x}_2(0),\tilde J_0)$,
  the restriction of $\bar f$ to any component of $\bar C$ is a simple map,
  and no irreducible component of $\bar C$ is entirely mapped into $V$.
  Moreover, the curve $\bar f|_{C_-}$ is irreducible and $\bar f|_{C_-}$ is an element of
  $\cl^{0,0,0,0}(p^!d,\tilde{\undl x}_2(0)\cap\overline{\tilde X}^-,\tilde J_0)$.
  The map $\bar f$ is the limit of a unique element of
  $\cl(p^!d, \tilde{\undl x}_2(\lambda),\tilde J_\lambda)$ as $\lambda$ goes to $0$.

\end{prop}

\begin{prop}\label{prop:welrep22}
Let $X_\rb$ be a compact real symplectic $4$-manifold,
$d\in H_2(X;\zb)$ such that $c_1(X)\cdot d>0$ and $\tau_*d=-d$.
Suppose $y_1$, $y_2\in (X\setminus\rb X)$ is a $\tau$-conjugated pair.
Denote by $p:X_{0,1}\to X$ the projection of the real symplectic blow-up of $X$ at $y_1$, $y_2$.
Then
$$
W_{X_{0,1}}(p^!d,s) =\sum_{C_-\in\rb\cl^{0,0,0,0}
(p^!d,\tilde{\undl x}_2(0)\cap\tilde X^-,\tilde J_0)}
(-1)^{m_{X_{0,1}}(C_-)}.
$$
Moreover, if $s\geq 1$, then
\begin{eqnarray*}
& & W_{X_{0,1}}(p^!d-[E_1]-[E_2],s-1) \\
& & \\
  & =  & \sum_{C_-\in\rb\cl^{0,\delta_1^c,0,\delta_1^c}
(p^!d-[E_1]-[E_2],\tilde{\undl x}_1(0)\cap\tilde X^-,\tilde J_0)}
(-1)^{m_{X_{0,1}}(C_-)},
\end{eqnarray*}
where $E_1$, $E_2$ are the exceptional divisors.
\end{prop}

\begin{rem}
Proposition $\ref{prop:welrep21}$ and Proposition $\ref{prop:welrep22}$ implies
Theorem $\ref{thm:bpc}$.
\end{rem}

\section{Wall-crossing formula of Welschinger invariants}

\subsection{Wall-crossing formula}
Welschinger \cite{wel2005a}  introduced a new invariant $\theta_{X_{\mathbb{R}}}(d,s)$ to describe the variation
of Welschinger invariants when replacing a pair of real fixed points in the same component of $\rb X$
by a pair of $\tau$-conjugated points. Welschinger proved the following wall-crossing formula, Theorem 3.2 in \cite{wel2005a},

\begin{thm}(\cite{wel2005a})\label{thm:welwcf}
Let $(X,\omega,\tau)$ be a compact real symplectic $4$-manifold such that
$\rb X$ is connected, $d\in H_2(X;\zb)$ such that $c_1(X)\cdot d -1>0$ and $\tau_* d=-d$, and $s$ be an integer between
$1$ and $[\frac{c_1(X)\cdot d-1}{2}]$.
Then
$$
W_{X_\rb}(d,s-1)=W_{X_\rb}(d,s)+2\theta_{X_\rb}(d,s-1).
$$
\end{thm}

In algebraic geometry category, Itenberg, Kharlamov and Shustin \cite{iks2004} observed that The invariant $\theta_{X_{\mathbb{R}}}(d,s)$ can be considered as the Welschinger invariants on the
blow-up at the fixed real point.  In the following, we will use the degeneration technique to verify this observation for any symplectic $4$-manifolds.

Perform the real symplectic cut on $X$ at the real point $x\in\rb X$.
We can get
$$
\bar{X}^+\cong\pb^2\,\,\,\, \text{ and }\,\,\,\, \bar{X}^-\cong X_{1,0}.
$$
In this section, we assume that $d\in H_2(X,\mathbb{Z})$ such that $c_1(X)\cdot d \geq 4$ and $\tau_*d = -d$. Denote $\pi:\zl\to\Delta$, $\undl x(\lambda)$, $J$, $X_\sharp$,
$\cl(d,\undl x(0),J_0)$, $C_*$, $*=+,-$, as in subsection $\ref{subsec:bpform}$. First of all, we have

\begin{prop}\label{prop:wall1}
Assume that $\undl x(0)\cap \bar{X}^+=\{p_1,p_2\}$,
and $\undl x(0)\cap \bar{X}^-\neq\emptyset$.
Then for a generic $J_0$,
the set $\cl(d,\undl x(0),J_0)$ is finite, and only depends on $\undl x(0)$ and $J_0$.
Given an element $\bar f:\bar C\to X_\sharp$  of $\cl(d,\undl x(0),J_0)$,
the restriction of $\bar f$ to any component of $\bar C$ is a simple map,
and no irreducible component of $\bar C$ is entirely mapped into $V$.
Moreover,

(1) $C_+$ is irreducible and $\bar f(C_+)$ realizes the class $[H]$
    passing through $\{p_1,p_2\}$.
   The curve $C_-$ is irreducible and $\bar f|_{C_-}$ is an element of
    $\cl^{\delta_1,0}(p^!d-[E],\undl x(0)\cap \bar{X}^-\sqcup\{q\},J_0)$
    for some $q\in V$.
    The map $\bar f$ is the limit of a unique element of
    $\cl(d,\undl x(\lambda),J_\lambda)$ as $\lambda$ goes to $0$.

(2) $C_+$ has exactly two irreducible components and
    the image of each component realizes the class $[H]$
    and passing through one point of $\{p_1,p_2\}$.
    The curve $C_-$ is irreducible and $\bar f|_{C_-}$ is an element of
    $\cl^{0,2\delta_1}(p^!d-2[E],\undl x(0)\cap \bar{X}^-,J_0)$.
    The map $\bar f$ is the limit of a unique element of
    $\cl(d,\undl x(\lambda),J_\lambda)$ as $\lambda$ goes to $0$.

\end{prop}

\begin{proof}
Example 11.4 and Lemma 14.6 in \cite{ip2004} implies that no component of $\bar C$ is entirely mapped into $V$.
In the real blow-up $\bar{X}^-\cong X_{1,0}$, $[E]^2=-1$. The adjunction formula implies that $c_1(\bar{X}^-)\cdot[E]=1$.
Suppose $\bar f_*[C_+]=a[H]$, $\bar f_*[C_-]=p^!d-b[E]$.
 Thus by considering a representative of $V$ in $\bar{X}^+$ and another in $\bar{X}^-$ respectively,
we have
$$
a=\bar f_*[C_+]\cdot[H]=(p^!d-b[E])\cdot[E]=p^!d\cdot[E]+b=b.
$$

Since $\undl x(0)\cap \bar{X}^+=\{p_1,p_2\}$, so $\bar f(C_+)$ passes through the two points $p_1, p_2$. Then $\underline{x}(0)\cap \bar{X}^-\not=\emptyset$
  implies that $a=b\glt1$ and $c_1(X)\cdot d\geq 4$. Therefore,
$\bar f(C_-)$ passes through all the $c_1(X)\cdot d-3$ points in $\undl x(0)\cap \bar{X}^-$
and realizes the class $p^!d-b[E]$ in $H_2(\bar{X}^-;\zb)$.

 Suppose  that $C_{-}$
consists of irreducible components $\{C_{-i}\}_{i=1}^m$
with $0\llt k\llt m$ irreducible components $\{C_{-i}\}_{i=1}^k$ such that
the restriction $\bar f|_{C_{-i}}$, $i=1,...,k$, is non-simple
which factors through a non-trivial ramified covering of degree $\delta_i\glt2$
of a simple map $f_i:\pb^1\to \bar{X}^-$.
Assume that $(f_i)_*[\pb^1]=d_i$, $i=1,...,k$, and $\bar f_*[C_{-j}]=d_j$, $j=k+1,...,m$.
Then $\sum_{i=1}^k\delta_id_i+\sum_{j=k+1}^{m} d_j=p^!d-b[E]$.
\begin{eqnarray*}
& & c_1(\bar{X}^-)\cdot(\sum_{i=1}^kd_i)-k+c_1(\bar{X}^-)\cdot(\sum_{j=k+1}^md_j)-(m-k)\\
& & \\
&  & \geq  c_1(X)\cdot d -3\\
& & \\
& & = c_1(\bar{X}^-)\cdot(\sum_{i=1}^k\delta_id_i+\sum_{j=k+1}^md_j)+b-3.
\end{eqnarray*}
Therefore,
\begin{equation}\label{4-1}
\sum_{i=1}^k(1-\delta_i)c_1(\bar{X}^-)\cdot d_i
\glt m+b-3.
\end{equation}
Since $c_1(\bar{X}^-)\cdot d_i\glt0$, so we have
$$
     m+b\leq 3.
$$

First of all, we assume that $k\geq 1$. Then (\ref{4-1}) implies $m+b< 3$. Thus we have $m=b=1$. This implies
that $C_-$ has one component. Furthermore, assume that $\bar{f}|_{C_-}$ factors through a non-trivial ramified covering of degree $\delta\geq 2$
of a simple map $\bar{f}_-: \mathbb{P}^1\longrightarrow \bar{X}^-$. Then we have
\begin{equation*}
\frac{1}{\delta}c_1(\bar{X}^-)\cdot(p^!d-b[E])-1 \glt c_1(X)\cdot d-3.
\end{equation*}
Therefore,
$$
 c_1(X)\cdot d+2\delta-\delta c_1(X)\cdot d \glt b.
$$
Since $\delta\glt2$ and  $c_1(X)\cdot d\glt4$. So $b\llt0$. This is in contradiction with $b=1$.
This contradiction implies that $k=0$.

Next, we assume that  $k=0$. (\ref{4-1}) implies that we only need to consider the following two cases.

{\bf Case I}: $m=2, b=1$.

In this case, $\bar f|_{C_+}$ is constrained by $\{p_1,p_2\}$ and
$\bar f|_{C_+}\in\cl^{0,\delta_1}([H],\{p_1,p_2\},J_0)$.
Therefore $\bar f|_{C_-}$ has to pass the point of intersection of $\bar f|_{C_+}$ and $V$
which is distinct to $\undl x(0)\cap \bar X^-$.
$\bar f_{C_-}$ will pass $c_1(X)\cdot d-2 = c_1(\bar{X})\cdot (p^!d- [E]) -1 $ distinct points
which implies that $\bar f_{C_-}$ is irreducible. This is in contradiction with that $C_-$ has two components.
This  contradiction implies that this case is impossible.

{\bf Case II}: $m=1, b=1$ or $2$.

If $b=a=1$, $C_+$ must have exactly $1$ component
and its image under $\bar{f}$ realizes the class $[H]$.
$\bar f|_{C_+}$ is simple because of Lemma $\ref{lem:simap}$.
 $\bar f|_{C_+}\in\cl^{0,\delta_1}([H],\{p_1,p_2\},J_0)$.
By the positivity of intersection, there only has one curve in $\cl^{0,\delta_1}([H],\{p_1,p_2\},J_0)$
which is an embedded simple curve.
Denote by $q$   the point of intersection of $\bar f|_{C_+}$ and $V$.
The point $q$   depends only on $\cl^{0,\delta_1}([H],\{p_1,p_2\},J_0)$.
Therefore  $\bar f|_{C_-}$ has to pass $\undl x(0)\cap \bar{X}^-\sqcup\{q\}$, and 
$\bar f|_{C_-}$ is an element of $\cl^{\delta_1,0}(p^!d-[E],\undl x(0)\cap \bar{X}^-\sqcup\{q\},J_0)$.

If $b=a=2$, $\bar f|_{C_-}$ is an element of $\cl^{0,2\delta_1}(p^!d-2[E],\undl x(0)\cap \bar{X}^-,J_0)$.
$\bar f|_{C_-}$ intersects $E$ transversely in $2$ distinct points.
Note that the curve $\bar C$ is rationl and any component of $\bar{f}(C_+)$  intersects $E$ in $\bar{X}^+$,
so   $C_{+}$ has exactly $2$ irreducible components.
Furthermore each component of $\bar{f}(C_+)$  realizes $[H]$ and passes through one point of $\{p_1,p_2\}$.

  The rest of the Proposition can be proved similar to Proposition $\ref{prop:sum11}$.
We omit it here.
\end{proof}

\begin{prop}\label{prop:wall2}
Let $X_\rb$ be a compact real symplectic $4$-manifold,
$d\in H_2(X;\zb)$ such that  $c_1(X)\cdot d \glt 4 $
and $\tau_*d=-d$.
Denote by $p: X_{1,0}\to X$ the projection of the real symplectic blow-up of $X$ at   $x\in\rb X$.
Then if $s\glt1$, we have
 
\begin{eqnarray}\label{eq:wall2-1}
W_{X_\rb}(d,s-1)& = & \sum_{C_1\in\rb\cl^{\delta_1^r,0}(p^!d-[E],\undl x(0)\cap \bar{X}^-\sqcup\{q\},J_0)}(-1)^{m_{X_{1,0}}(C_1)}\nonumber\\
 & &\\
& &+2\sum_{C_2\in\rb\cl^{0,2\delta_1^r}(p^!d-2[E],\undl x(0)\cap \bar{X}^-,J_0)}(-1)^{m_{X_{1,0}}(C_2)}, \nonumber
\end{eqnarray}
\begin{eqnarray}\label{eq:wall2-2}
W_{X_\rb}(d,s) & = & \sum_{C_1\in\rb\cl^{\delta_1^r,0}(p^!d-[E],\undl x(0)\cap \bar{X}^-\sqcup\{q\},J_0)}(-1)^{m_{X_{1,0}}(C_1)}\nonumber\\
& & \\
&  & -2\sum_{C_2\in\rb\cl^{0,2\delta_1^c}(p^!d-2[E],\undl x(0)\cap \bar{X}^-,J_0)}(-1)^{m_{X_{1,0}}(C_2)},\nonumber
\end{eqnarray}
where $E$ is the exceptional divisor and $q$ is some particular point in $V$.
 
\end{prop}

\begin{proof}
Equip the small disc $\Delta$ with the standard complex conjugation.
From subsection $\ref{subsec:realsympcut}$,
we know one can equip the symplectic sum $\pi:\zl\to\Delta$ with a real structure $\tau_{\zl}$
which is induced by the real structures $\tau_-$, $\tau_+$
on the real symplectic cuts $\bar{X}^-$ and $\bar{X}^+$
such that the map $\pi:\zl\to\Delta$ is real.
Choose a set of real sections $\undl x:\Delta\to\zl$ such that $\underline{x}(0)\cap\bar{X}^+_{1,0}$ contains two points.
Let $\bar f:\bar C\to X_\sharp$ be a real element in $\rb\cl(d,\undl x(0),J_0)$.

Next, we will divide the proof into two cases according to the type of real configuration points.

{\bf Case I}; $\undl x(0)\cap \bar{X}^+=\{p_1,p_2\}$ and $p_1$, $p_2\in\rb X$.

From Proposition $\ref{prop:wall1}$, there are two types of the limited curve $\bar{f}$ as following.

{\bf  Type I-1}:  $C_+$ has only one component.
 
$\bar f_*[C_+]=[H]$ and $\bar f|_{C_+}\in\cl^{0,\delta_1}([H],\{p_1,p_2\},J_0)$ is an embedded simple curve.
The intersection point $q$ of $\bar{f}(C_+) $ with $V$, determined by $\cl^{0,\delta_1}([H],\\ \{p_1,p_2\},J_0)$, has to be real.
In this case, since $\bar f(C_+)$ has no self-intersection point, so $\bar f(C_+)$ has no node.
Therefore, there is only one possibility  to recover a real curve $\bar f(\bar C)$ from $\bar f|_{C_+}$
when $\bar f|_{C_-}$ is fixed. So we have
\begin{eqnarray}\label{4-2}
m_{X_\sharp}(\bar f(\bar C)) & = & m_{\bar{X}^-}(\bar f|_{\bar C_-})+m_{\bar{X}^+}(\bar f|_{C_+})\nonumber\\
 & & \\
 & = & m_{\bar{X}^-}(\bar f|_{\bar C_-}). \nonumber
\end{eqnarray}

{\bf Type I-2}: $C_+$ has exactly two irreducible components $C_{+i}$, $i=1,2$.
In this case, $\bar f_*[C_{+i}]=[H]$ and $\bar f|_{C_{+i}}$ is an embedded simple curve.
 By the positivity of intersections, $\bar f(C_{+1})$ intersects
$\bar f(C_{+2})$ at one point. This point has to be a real node of $\bar f(C_+)$
because $\bar f(C_+)$ is real.
Since $\bar f(C_{+i})$ passes $p_i\in\rb X$,
$\bar f(C_{+1})$ and $\bar f(C_{+2})$ can not be two $\tau$-conjugated components.
Therefore the real nodal point of $\bar f(C_+)$ has to be non-isolated.
Moreover, each $\bar f(C_{+i})$ intersects $V$ at a real point $q_i$ transversally
and $\bar f_{C_{+i}}\in\cl^{\delta_1^r,0}([H],\{p_i\}\sqcup\{q_i\},J_0)$.
 From Proposition $\ref{prop:wall1}$ $(2)$, we can get
the curve $C_-$ is irreducible and $\bar f|_{C_-}$ is an element of
$\cl^{0,2\delta_1}(p^!d-2[E],\undl x(0)\cap \bar{X}^-,J_0)$.
$\bar f|_{C_+}$ and $\bar f|_{C_-}$ form the limited curve $\bar f$.
We know $\bar f(C_-)$ intersects $V$ at two real non-prescribed points transversally.
Therefore $\bar f|_{C_-}\in\cl^{0,2\delta_1^r}(p^!d-2[E],\undl x(0)\cap \bar{X}^-,J_0)$.
Therefore, there are two possibilities   to recover a real curve $\bar f(\bar C)$ from $\bar f|_{C_+}$
  when $\bar f|_{C_-}$ is fixed.
  We have
  \begin{eqnarray}\label{4-3}
  m_{X_\sharp}(\bar f(\bar C)) & = & m_{\bar{X}^-}(\bar f|_{\bar C_-})+m_{\bar{X}^+}(\bar f|_{C_+})\nonumber\\
  & & \\
   & = & m_{\bar{X}^-}(\bar f|_{\bar C_-}). \nonumber
  \end{eqnarray}

By Proposition $\ref{prop:wall1}$, an element $\bar f$ of $\cl(d,\undl x(0),J_0)$ is the limit
of a unique element of $\cl(d,\undl x(\lambda),J_\lambda)$ as $\lambda$ goes to $0$. So the latter has to be real when
$\bar f$ is real and $\lambda\in\cb^*$ is small. When deforming $\bar f$, no node appears
in a neighborhood of $V\cap\bar f(\bar C)$. From the analysis of Case I,
we know the elements of $\cl(d,\undl x(0),J_0)$ have two different types.
Therefore the elements of $\rb\cl(d,\undl x(\lambda),J_\lambda)$ will degenerate into
two types: in type I-1,
  an element of $\rb\cl^{\delta_1^r,0}(p^!d-[E],\undl x(0)\cap \bar{X}^-\sqcup\{q\},J_0)$
corresponds to a unique element of the limited curve;
in type I-2, an element of
$\rb\cl^{0,2\delta_1^r}(p^!d-2[E],\undl x(0)\cap \bar{X}^-,J_0)$ corresponds
to two limited curves with the same mass. One can get formula $(\ref{eq:wall2-1})$
easily from (\ref{4-2}) and (\ref{4-3}).

{\bf Case II}:  
 $\undl x(0)\cap X^+=\{p,p'\}$ with $p$, $p'\in X\setminus\rb X$ and $\tau(p_1) = p_2$.

From Proposition $\ref{prop:wall1}$, we also obtain  two types of the limited curve $\bar f$ as following.

{\bf Type II-1}:  $C_+$ has only one component.
 
$\bar f_*[C_+]=[H]$ and $\bar f|_{C_+}\in\cl^{0,\delta_1}([H],\{p,p'\},J_0)$ is an embedded simple curve.
The $\tau$-conjugated pair $p,p'\in X\setminus\rb X$ can be chosen such that
the intersection point, determined by $\cl^{0,\delta_1}([H],\{p,p'\},J_0)$ , is also $q$.
So $\bar f|_{C_-}$ belongs to
  $\cl^{\delta^r_1,0}(p^!d-[E],\undl x(0)\cap \bar{X}^-\sqcup\{q\},J_0)$,
  and the remaining argument is the same as that in the case I.
  We omit it here. 

{\bf Type II-2}: $C_+$ has exactly two irreducible components $C_{+i}$,$i=1,2$.
 In this case, $\bar f_*[C_{+i}]=[H]$ and $\bar f|_{C_{+i}}$ is an embedded simple curve.
 By the positivity of intersections, $\bar f(C_{+1})$ intersects
$\bar f(C_{+2})$ at one real point which is a real node of $\bar f(C_+)$.
Since $\bar f(C_{+i})$ passes $p_i\in X\setminus\rb X$ with $\tau(p_1)=p_2$,
$\bar f(C_{+1})$ and $\bar f(C_{+2})$ are two $\tau$-conjugated components.
Therefore the real nodal point of $\bar f(C_+)$ has to be isolated.
Moreover, each $\bar f(C_{+i})$ intersects $V$ at a point $q_i$ transversally
with $\tau(q_1)=q_2$.
 From Proposition $\ref{prop:wall1}$ $(2)$, we can get
the curve $C_-$ is irreducible and $\bar f|_{C_-}$ is an element of
$\cl^{0,2\delta_1}(p^!d-2[E],\undl x(0)\cap \bar{X}^-,J_0)$.
$\bar f|_{C_+}$ and $\bar f|_{C_-}$ form the limited curve $\bar f$.
We know $\bar f(C_-)$ intersects $V$ at two $\tau$-conjugated non-prescribed points transversally.
Therefore $\bar f|_{C_-}\in\cl^{0,2\delta_1^c}(p^!d-2[E],\undl x(0)\cap \bar{X}^-,J_0)$.
There are two possibilities  to recover a real curve $\bar f(\bar C)$ from $\bar f|_{C_+}$
when $\bar f|_{C_-}$ is fixed.   We have
\begin{eqnarray}\label{4-4}
  m_{X_\sharp}(\bar f(\bar C)) & = & m_{\bar{X}^-}(\bar f|_{\bar C_-})+m_{\bar{X}^+}(\bar f|_{C_+})\nonumber\\
  & & \\
  & = & m_{\bar{X}^-}(\bar f|_{\bar C_-})+1. \nonumber
\end{eqnarray}

  By Proposition $\ref{prop:wall1}$, an element $\bar f$ of $\cl(d,\undl x(0),J_0)$ is the limit
of a unique element of $\cl(d,\undl x(\lambda),J_\lambda)$ as $\lambda$ goes to $0$. So the latter has to be real when
$\bar f$ is real and $\lambda\in\cb^*$ is small. When deforming $\bar f$, no node appears
in a neighborhood of $V\cap\bar f(\bar C)$. From the analysis of Case II,
we know the elements of $\cl(d,\undl x(0),J_0)$ have two different types.
Therefore the elements of $\rb\cl(d,\undl x(\lambda),J_\lambda)$ will degenerate into
two types: in type II-1,
  an element of $\rb\cl^{\delta_1^r,0}(p^!d-[E],\undl x(0)\cap \bar{X}^-\sqcup\{q\},J_0)$
corresponds to a unique element of the limited curve;
in type II-2, an element of
$\rb\cl^{0,2\delta_1^c}(p^!d-2[E],\undl x(0)\cap \bar{X}^-,J_0)$ corresponds
to two limited curves with the same mass. One can get formula $(\ref{eq:wall2-2})$
easily from (\ref{4-2}) and (\ref{4-4}).
\end{proof}

Next, perform the real symplectic cut along the exceptional divisor $E$ in $X_{1,0}$.
We can get $\overline{ X}^+_{1,0}\cong \mathbb{P}_E(\mathcal{O}(-1)\oplus\mathcal{O})$, $\overline{  X}^-_{1,0}\cong X_{1,0}$, $V$ as in subsection $\ref{subsec:bpform}$.

Let $\tilde\zl$ be the real symplectic sum of $\overline{ X}^+_{1,0}$ and $\overline{ X}^-_{1,0}$ along $V$.
Let $p^!d-2[E]\in H_2(\tilde\zl_\lambda;\zb)$, where $d\in H_2(X;\zb)$.
Choose $\tilde{\undl x}(\lambda)$
a set of $c_1(X)\cdot d-3$ real sections $\Delta\to\tilde\zl$ such that
$\tilde{\undl x}(0)\cap V=\emptyset$, $\tilde{\undl x}(0)\cap\overline{ X}^+_{1,0}=\emptyset$.
Choose an almost complex structure $\tilde J$ on $\tilde\zl$ as before.
Denote $\tilde X_\sharp$, $\cl(p^!d-2[E],\tilde{\undl x}(0),\tilde J_0)$,
 $C_*$, $*=+,-$, as  in subsection $\ref{subsec:bpform}$.

\begin{prop}\label{prop:wall3}
For a generic $\tilde J_0$, the set $\cl(p^!d-2[E],\tilde{\undl x}(0),\tilde J_0)$ is finite,
and only depends on $\tilde{\undl x}(0)$ and $\tilde J_0$.
Given an element $\bar f:\bar C\to\tilde X_\sharp$  of
$\cl(p^!d-2[E],\tilde{\undl x}(0),\tilde J_0)$,
the restriction of $\bar f$ to any component of $\bar C$ is a simple map,
and no irreducible component of $\bar C$ is entirely mapped into $V$.
Moreover, the curve $C_-$ is irreducible and $\bar f|_{C_-}$ is an element of
$\cl^{0,2\delta_1}(p^!d-2[E],\tilde{\undl x}(0)\cap \overline{ X}^-_{1,0},\tilde J_0)$.
The curve $C_+$ has two irreducible components.
Each component of $\bar{f}(C_+)$ realizes the fiber class $[F]$ in $\pb_E(\ol(-1)\oplus\ol)\to E$.
The map $\bar f$ is the limit of a unique element of
$\cl(p^!d-2[E],\tilde{\undl x}(\lambda),\tilde J_\lambda)$ as $\lambda$ goes to $0$.
\end{prop}

\begin{proof}
As before, we know that no component of $\bar C$ is entirely mapped into $V$.
  Since $\bar f_*[\bar C]=p^!d-2[E]$, we may
suppose $\bar f_*[C_+]=a[F]+b[E_\infty]$, $a, b\geq 0$, $\bar f_*[C_-]=p^!d -k[E]$, $k\glt 0$ ,
where $F$ is a fiber of $\pb_E(\ol(-1)\oplus\ol)\to E$ with $F\cdot[E_0]=1$ and $F\cdot[E_\infty]=1$. Then
\begin{eqnarray*}
a+b=(a[F]+b[E_\infty])\cdot[E_\infty] & = & (p^!d-k[E])\cdot[E]=k,\\
& & \\
a=(a[F]+b[E_\infty])\cdot[E_0] & = & (p^!d-2[E])\cdot[E]=2.
\end{eqnarray*}
This implies $k\geq 2$.

In $\overline{ X}^-_{1,0}$, we know that $\bar f|_{C_-}$ passes through
$$
c_1(X)\cdot d-3=c_1(\overline{ X}^-_{1,0})\cdot (p^!d-2[E])-1
$$
distinct points.  The same argument as in the proof of Proposition \ref{prop:sum11} shows that $C_-$ is irreducible.

Assume that $\bar f|_{C_{-}}$ is non-simple. Then $\bar f|_{C_{-}}$
factors through a non-trivial ramified covering of degree $\delta \geq 2$
of a simple map $f_0:\pb^1\to\overline{\tilde X}^-$. Then $(f_0)_*[\pb^1]=\frac{1}{\delta}(p^!d-k[E])$. Therefore,
\begin{equation*}
\frac{1}{\delta}c_1(\overline{\tilde X}^-)\cdot(p^!d-k[E])-1   \glt  c_1(  X )\cdot d -3.
\end{equation*}
This implies
\begin{equation*}
c_1(X)\cdot d-\delta c_1(X)\cdot d+2\delta  \glt   k.
\end{equation*}
Since $\delta\glt2$, $c_1(X)\cdot d\glt4$, so we have $k\leq 0$. This is in contradiction with $k\geq 2$. Thus $\bar f|_{C_{-}}$ is simple.

On the other hand, we have
\begin{eqnarray*}
c_1(\tilde X^-)\cdot(p^!d -k[E])-1 & = & c_1( X)\cdot d -k-1\\
& \glt & c_1(  X )\cdot d -3.
\end{eqnarray*}
This implies $k\leq 2$. Therefore, we have $k=2$, $b=0$.
Since the image of $C_-$ under $\bar{f}$ intersects $V$ transversally in $2$ distinct points.
Therefore, $C_+$ has two irreducible components $C_{+i}$ such that
$\bar f_*[C_{+i}]=[F]$, $i=1,2$.

  The rest of the Proposition can be obtained by the
similar argument in the proof of Proposition $\ref{prop:sum11}$.
We omit it here.
\end{proof}

\begin{prop}\label{prop:wall4}
Let $X_\rb$ be a compact real symplectic $4$-manifold,
 $d\in H_2(X;\zb)$ such that $c_1(X)\cdot d\glt4$ and $\tau_*d=-d$.
Denote by $p: X_{1,0}\to X$ the projection of the real symplectic blow-up of $X$ at   $x\in\rb X$.
Then if $s\glt1$,
\begin{eqnarray*}
W_{X_{1,0}}(p^!(d)-2[E],s-1)
& = & \sum_{C_-\in\rb\cl^{0,2\delta_1^r}(p^!d-2[E],\tilde{\undl x}_3(0)\cap\tilde X^-,\tilde J_0)}(-1)^{m_{X_{1,0}}(C_-)}\\
& & \\
& & +\sum_{C_-\in\rb\cl^{0,2\delta_1^c}(p^!d-2[E],\tilde{\undl x}_3(0)\cap\tilde X^-,\tilde J_0)}(-1)^{m_{X_{1,0}}(C_-)},
\end{eqnarray*}
where $E$ is the exceptional divisor.
\end{prop}

\begin{rem}
The proof of Proposition $\ref{prop:wall4}$  is
similar to Proposition $\ref{prop:welrep11}$.
Proposition $\ref{prop:wall2}$ and Proposition $\ref{prop:wall4}$ imply Theorem $\ref{thm:wall}$.
\end{rem}

\subsection{Generating function}
In this subsection,
we restate our formulae in the form of generating functions. Denote by
$$
W_{X_\rb,L,F}^d(T)=\sum_{s=0}^{[\frac{c_1(X)\cdot d-1}{2}]}W_{X_\rb,L,F}(d,s)T^s\in\zb[T]
$$
the generating function of Welschinger invariants which encodes all the information of the Welschinger invariants.

Let $X_\rb$ be a compact real symplectic $4$-manifold.
If $\rb X$ is disconnected, the previous formulae are still true,
and can be proved in the same method as Theorem $\ref{thm:wall}$ and Corollary $\ref{cor:bp}$.
Suppose $\undl{x}'\subset X$ is a real set consisting of $r'$ points in $L$ and
$s'$ pairs of $\tau$-conjugated points in $X$ with $r'+2s'\llt c_1(X)\cdot d-1$.
We denote the connected component of $\rb X_{r',s'}$ corresponding to $L$ by $\tilde L$.
If there is only one blown-up real  point in $L$, $\tilde L=L\sharp\rb P^2$.
We assume $F$ has a $\tau$-invariant compact representative $\fl\subset X\setminus\undl{x}'$,
and denote $\tilde F=p^!F$.
Denote by $p:X_{r',s'}\to X$ the projection of the real symplectic blow-up of $X$ at $\undl{x}'$.
Then
\begin{eqnarray*}
& & W_{X_\rb,L,F}^d(T)=W_{X_{r',s'},\tilde L,\tilde F}^{p^!d}(T),\\
& & \\
 & & W_{X_\rb,L,F}^d(T)-W_{X_\rb,L,F}(d,0)-...-W_{X_{\rb},L,F}(d,s'-1)T^{s'-1}\\
& & =W_{X_{r',s'},\tilde L,\tilde F}^{p^!d-\sum_{i=1}^{r'}[E_i]-\sum_{j=1}^{s'}([E_j']+[E_j''])}(T)T^{s'},\\
& & \\
& & W_{X_\rb,L,F}^d(T)T=W_{X_\rb,L,F}^d(T)-W_{X_\rb,L,F}(d,0)+2W_{X_{1,0},\tilde L,\tilde F}^{p^!d-2[E]}(T)T,
\end{eqnarray*}
where $E_i$, $E_j'$, $E_j''$ denote the exceptional divisors
corresponding to the real set $\undl x'$ respectively.

\section{Real enumeration}

In this section we will give some applications of the blow-up formula of Welschinger
invariants.

\subsection{Blow-up of $\mathbb{C}P^2$}
Let $\mathbb{C}P^2_{r,s}$ denote the blow-up of $\mathbb{C}P^2$ at $r$ real points
and $s$ pairs of conjugated points.
The projective plane with the standard symplectic structure and complex conjugate
is a real symplectic manifold. \cite{abl2011} proved a recursive formula of
Welschinger invariants in the projective plane.
   Using the results of \cite{abl2011} and the blow-up formula of Welschinger invariants
(Corollary $\ref{cor:bp}$), we can compute the invariants of $\mathbb{C}P^2_{r,s}$.

\begin{center}

\begin{tabular}[c]{|c|l|l|l|l|l|}
\hline
s & 0 & 1 & 2& 3& 4\\
\hline
$W(c_1(X),0)$ & 8  & 6  &  4 & 2 & 0 \\
\hline
$W(c_1(X),1)$ & 6  &  4 & 2 & 0 & -\\
\hline
$W(c_1(X),2)$ &  4 & 2  & 0 & -  & -\\
\hline
$W(c_1(X),3)$ & 2  & 0  & - & - & -\\
\hline
\end{tabular}

\vskip 0.5cm

\begin{tabular}{|c|c|c|c|}
  \hline
    & $W([H],0)$ & $W(2[H],0)$  & $W(4[H],0)$ \\
  \hline
  $X=\mathbb{C}P^2_{3,0}$ & 1 & 1  & 240 \\
  \hline
  $X=\mathbb{C}P^2_{1,1}$ & 1 & 1  & 144 \\
  \hline
\end{tabular}

\vskip 0.5cm

{Table 1. Welschinger invariants of $\mathbb{C}P^2_{r,s}$ with $r+2s\llt8$}

\end{center}

Note that the Welschinger invariants of $\mathbb{C}P^2_{r,s}$ with purely
real point constraints were computed in \cite{iks2003,iks2004,iks2005,iks2009,iks2013a,iks2013b,iks2013c}.
And the Welschinger invariants of $\mathbb{C}P^2_{r,s}$ were totally computed in \cite{hs2012}.
The invariants for $r+2s\llt3$ with arbitrary real and conjugated pairs of point constraints
were studied by \cite{bm2008}. For $6\llt r+2s\llt8$,
$W_{\mathbb{C}P^2_{r,s}}(d,s')$ were considered in \cite{bru2015} with any point constraints.

\subsection{Conic bundles and Del Pezzo surfaces of degree 2}

Recall that there are $12$ topological types of degree $2$
real Del Pezzo surfaces.
I. Itenberg, V. Kharlamov, and E. Shustin \cite{iks2013a} computed
the Welschinger invariants with purely real point constraints of all
degree $2$ real Del Pezzo surfaces.
E. Brugall\'e \cite{bru2015}, computed the Welschinger invariants with arbitrary point
constraints of real degree $2$ Del Pezzo surfaces with a non-orientable real part.
A. Horev and J. Solomon \cite{hs2012} also computed Welschinger invariants with
arbitrary point constraints of some degree $2$ Del Pezzo surfaces with a non-orientable real part.
E. Brugall\'e \cite{bru2015,brugalle2016} computed the Welschinger invariants of all
the real degree $1$ Del Pezzo surface.
Since every degree $2n-1$ Del Pezzo surface which is not the minimal Del Pezzo surface
is the blow-up of a degree $2n$ Del Pezzo surface at a real point.
We can use the blow-up formula to compute the Welschinger invariants with conjugated point
constraints in the remaining
five topological types of degree $2$ real Del Pezzo surfaces
with no non-orientable real part.

Let $\bb^{n}$ be the real conic bundle with $2n$ singular fibers
and $X^1$ be the minimal real del Pezzo surface of degree $2$.
Endow $\pb^1\times\pb^1$ with the standard real structure.
So $X^1$, $\bb^3$, $\bb^2_{0,1}$, $\bb^1_{0,2}$, $(\pb^1\times\pb^1)_{0,3}$
are the five topological types of real Del Pezzo surfaces of degree $2$
with real parts $\sqcup4S^2$ $\sqcup3S^2$, $\sqcup2S^2$, $S^2$, $S^1\times S^1$,
respectively.
The following tables are taken from \cite{bru2015,brugalle2016,iks2013a}.

\begin{center}
\begin{tabular}[c]{|c|c|c|c|c|c|}
\hline
  & $X^1$ & $\bb^3$ & $\bb^2_{0,1}$& $\bb^1_{0,2}$& $(\pb^1\times\pb^1)_{0,3}$\\
\hline
$W(2c_1(X),0)$ & 0  & 0  &  0 & 8 & 32   \\
\hline
\end{tabular}

\vskip 0.5cm
\begin{tabular}[c]{|c|c|c|c|c|c|}
\hline
  & $X^1_{1,0}$ & $\bb^3_{1,0}$ & $\bb^2_{1,1}$& $\bb^1_{1,2}$& $(\pb^1\times\pb^1)_{1,3}$\\
 \hline
$W(2c_1(X),0)$  & 18  & 10  &  6 & 6 & 10\\
 \hline
\end{tabular}
\end{center}
From Welschinger's wall-crossing formula:
$$
W_{X_\rb}(d,s-1)=W_{X_\rb}(d,s)+2W_{X_{1,0}}(p^!d-2[E],s-1),
$$
we can obtain the following values.
\begin{center}
\begin{tabular}[c]{|c|c|c|c|c|c|}
\hline
  & $X^1$ & $\bb^3$ & $\bb^2_{0,1}$& $\bb^1_{0,2}$& $(\pb^1\times\pb^1)_{0,3}$\\
\hline
$W(2c_1(X),1)$ & -36  & -20  &  -12 & -4 & 12 \\
\hline
\end{tabular}
\end{center}

\end{document}